\documentclass[preprint,12pt]{colt2024} 

\usepackage{bbm}
\usepackage{multirow}
\usepackage{enumitem}

\newcommand{\KL}{\mathrm{KL}}
\renewcommand{\H}{\mathrm{H}}

\newcommand{\I}{\mathrm{I}}
\newcommand{\uI}{\underline{\mathrm{I}}}
\newcommand{\R}{\mathbb{R}}
\newcommand{\E}{\mathbb{E}}

\renewcommand{\d}{\mathrm{d}}

\title[Information Lower Bounds for Robust Mean Estimation]{Information Lower Bounds for Robust Mean Estimation}
\usepackage{times}

\DeclareMathOperator{\Var}{\mathrm{Var}} 

\coltauthor{\Name{Rémy Degenne} \Email{remy.degenne@inria.fr}\AND
 \Name{Timothée Mathieu} \Email{timothee.mathieu@inria.fr}\\
 \addr Univ. Lille, Inria, CNRS, Centrale Lille, UMR 9189-CRIStAL, F-59000 Lille, France}

\begin{document}

\maketitle

\begin{abstract}%
We prove lower bounds on the error of any estimator for the mean of a real probability distribution under the knowledge that the distribution belongs to a given set.
We apply these lower bounds both to parametric and nonparametric estimation.
In the nonparametric case, we apply our results to the question of sub-Gaussian estimation for distributions with finite variance to obtain new lower bounds in the small error probability regime, and present an optimal estimator in that regime. 
In the (semi-)parametric case, we use the Fisher information to provide distribution-dependent lower bounds that are constant-tight asymptotically, of order $\sqrt{2\log(1/\delta)/(nI)}$ where $I$ is the Fisher information of the distribution. We use known minimizers of the Fisher information on some nonparametric set of distributions to give lower bounds in cases such as corrupted distributions, or bounded/semi-bounded distributions.
\end{abstract}


\section{Introduction}

Mean estimation is the task of finding an estimate of the expectation $\mathbb{E}_F[X]$ of a probability distribution $F$, given only access to $n \in \mathbb{N}$ independent samples $X_1, \ldots, X_n$.
In this work, $F$ will be a probability distribution on $\mathbb{R}$ ($F \in \mathcal P(\mathbb{R})$), with finite mean $\mathbb{E}_F[X]$.
We are interested in estimators, measurable functions $\tilde{\mu}_n : \mathbb{R}^n \to \mathbb{R}$ for a number of samples $n \in \mathbb{N}$, such that when $X_1, \ldots, X_n$ are i.i.d. with law $F$, the estimation error $\vert \tilde{\mu}_n(X_1, \ldots, X_n) - \mathbb{E}_F[X] \vert$ is as small as possible with probability close to 1.
We will also write simply $\tilde{\mu}_n$ for the random variable $\tilde{\mu}_n(X_1, \ldots, X_n)$.

\paragraph{The empirical mean}

Perhaps the simplest mean estimator is the empirical mean, $\hat{\mu}_n = \frac{1}{n}\sum_{i=1}^n X_i$. For any distribution $F$ with finite variance $\sigma^2$, by the central limit theorem,
\begin{align*}
\mathbb{P}\left( \vert \hat{\mu}_n - \mathbb{E}_F[X] \vert \ge \sqrt{\frac{\sigma^2}{n}} \Phi^{-1}(1 - \delta /2)\right) \to \delta \: .
\end{align*}
We are however interested in non-asymptotic guarantees, for a given sample size $n$.
For Gaussian random variables with variance $\sigma^2$, the empirical mean is also Gaussian and $\delta$ is a non-asymptotic upper bound of the deviation probability written above.
In particular, on Gaussian distributions the empirical mean has deviation bounded by $\sqrt{\frac{2\sigma^2 \log(2/\delta)}{n}}$ (a slightly weaker bound than $\sqrt{\frac{\sigma^2}{n}} \Phi^{-1}(1 - \delta /2)$ since $\Phi^{-1}(1 - \delta /2) \le \sqrt{2 \log (2/\delta)}$), without using any knowledge of the variance $\sigma^2$.
By Chernoff's method, we can extend the properties of the empirical mean to sub-Gaussian distributions: for any $\bar{\sigma}^2$-sub-Gaussian probability distribution, the empirical mean has error less than $\sqrt{\frac{2\bar{\sigma}^2 \log(2/\delta)}{n}}$ with probability $1 - \delta$, for any $\delta \in (0,1]$.

In contrast, lower bounds in \citep{catoni2012challenging} show that for the class of distributions with finite second moment, Chebyshev's inequality is essentially tight for the empirical mean \cite[Proposition 6.2]{catoni2012challenging}, and that on the other hand there exists other estimators with error width much smaller.
Then, the empirical mean is not a good estimator for that class of distributions and one needs to design more sophisticated estimators to achieve optimal rates of convergence for distribution with only finite moments bounds (and not an infinity of finite moments). 

\begin{definition}
An estimator $\tilde{\mu}_n$ is a $(x(n, \delta, F), \delta)$-estimator on $\mathcal D \subseteq \mathcal P(\mathbb{R})$ if for all $F \in \mathcal D$
\begin{align*}
F^{\otimes n}(\vert \tilde{\mu}_n - \mathbb{E}_{F}[X] \vert \ge x(n, \delta, F)) \le \delta
\: .
\end{align*}
\end{definition}

For a given set of distributions $\mathcal D \subseteq \mathbb{R}$, we will be interested in two main questions:
\begin{itemize}[nosep]
  \item Can we get lower bounds on $x(n, \delta, F)$ for a given set of distributions $\mathcal D$?
  \item Can we design estimators with error upper bounds that match those lower bounds?
\end{itemize}

\paragraph{Sub-Gaussian estimators}

Several recent works on mean estimation follow \cite{devroye2016sub} and consider the question of finding \emph{sub-Gaussian} estimators.
An estimator is said to be sub-Gaussian on a class of distributions $\mathcal D$ with constant $C$ if it is a $(C\sqrt{\frac{2\Var_F[X] \log(2/\delta)}{n}}, \delta)$-estimator on $\mathcal D$, where $\Var_F[X]$ is the variance of $F \in \mathcal D$.
The empirical mean is a sub-Gaussian estimator with constant 1 on the set of sub-Gaussian distributions, for any $\delta$.
\cite{devroye2016sub} showed that there is no sub-Gaussian estimator on the class of distributions with moment $1 + \alpha$ if $\alpha < 1$. The main class of interest in that line of work is then distributions with finite variance.

We have now several algorithms which are sub-Gaussian for a constant $C > 1$, for $\delta \in (e^{-c n}, 1)$ for some $c > 0$~\citep{lugosi2019mean,cherapanamjeri2019fast,minsker2021robust,cherapanamjeri2022optimal}. See also \citep{gobet2023robust} which provides an overview of these results.
On the other hand, using techniques based on median-of-means, \cite{lee2022optimal} and \cite{minsker2023efficient} designed algorithms that are sub-Gaussian for $C = 1 + o(1)$ (close to optimal since $C=1$ is the best attainable constant asymptotically), under hypotheses $\log(1/\delta)/n \to 0$ and $\delta \to 0$.
\cite{devroye2016sub} showed that the $\delta \ge e^{-cn}$ barrier is not improvable: for any $C$, there is no sub-Gaussian estimator with constant $C$ for $\delta \le e^{-cn}$, for $c$ small enough.
Our lower bounds shed light on the optimal deviation of an estimator in the $\delta \le e^{-cn}$ regime. We derive the optimal function $y(n,\delta)$ such that $(\sqrt{2 \Var_F[X] y(n, \delta)}, \delta)$-estimators exist. While $y(n, \delta) \approx \log(1/\delta)/n$ if $\log(1/\delta)/n$ is small enough, we also describe its shape for $\delta \ll e^{-n}$, where the notion of sub-Gaussian estimator with a given constant is not an adequate concept.

\subsection{Related work}
\label{sub:related_work}
In addition to the literature on sub-Gaussian estimators, our article is related to the following works.

\noindent\textit{Robust estimation.}
Originating in the works of~\cite{huber1964,huber1967behavior} and~\cite{10.1214/aoms/1177693054}, robustness theory was at first interested in providing efficient estimators of location parameters in the presence of outlier data, also called corruption. Most results were asymptotic and ignored the finite sample rates of convergence \citep{robuststat}. Mean estimation with sub-Gaussian rates  under heavy-tail assumptions is closely linked to robust estimation under corruption, although estimators that perform well for Heavy-tail distributions do not necessarily perform well on corrupted data. 


\noindent\textit{Information lower bounds.} The lower bounds we derive for estimation come from a reduction to testing, which follows \citep{tsybakov2009introduction}. We then bound the error probability of tests using information theoretic quantities like the Kullback-Leibler divergence ($\KL$).
This general idea was already present in the work of \cite{wald1945sequential}.
Lower bound techniques for tests using $\KL$ were then the subject of further development in the sequential testing literature \citep{hoeffding1953lower,hoeffding1960lower,lorden1983asymptotic}, and more recently in the multi-armed bandit literature \citep{lai1985asymptotically,burnetas1996optimal,garivier2016optimal}.
\cite{garivier2019explore} demonstrated that the data-processing inequality of $\KL$ is an easy-to-use and important tool to obtain lower bounds.

\noindent\textit{Lower bounds based on the Fisher information.} 
The classical bound by Cramer and Rao~\citep{cramer1946mathematical,radhakrishna1945information} says that in a sufficiently smooth model, the minimal variance an estimator can achieve is controlled by the inverse of the Fisher information. The Fisher information appear in our case through the asymptotics of the likelihood ratio (see Section 7.2 of \citep{le2000asymptotics}), which we use in Section~\ref{sec:parametric} to get the expression for the rates of convergence in a semi-parametric setting. A more recent line of work using the Fisher information can be found in~\citep{gupta2022finite,gupta2023finite} in which the authors obtain lower bound and upper bound on the estimation error with a dependency in the inverse of the Fisher information of a smoothed version of the target distribution. Compared to our work in Section~\ref{sec:parametric}, the lower bound provided in \cite[Section 6]{gupta2022finite} is asymptotic and valid only for target distributions that are smoothed versions of a distribution, for $\delta\ll 1$.
\vspace{-0.75em}

\subsection{Summary of our contributions}
\label{sub:contributions}

\begin{itemize}[noitemsep]
  \item We prove a generic lower bound on $x(n, \delta, F)$ for a $(x(n, \delta, F), \delta)$-estimator on any set of distributions with finite means (Theorem~\ref{thm:generic_delta_bound}).
  \item We apply that lower bound for $\mathcal D$ the set of distributions with finite variance (Theorem~\ref{thm:lower_bound_finite_var}) and we introduce an estimator with a matching upper bound for $\delta \ll 1/n$ (Theorem~\ref{thm:estimator_upper_bound}).
  \item We prove an asymptotic high probability lower bound for the estimation error in a semi-parametric setting using the Fisher information of the sampling distribution $\I(F)=\int (f')^2/f$ with $f$ the density of $F$ (Theorem~\ref{thm:without_2}), using quadratic mean differentiability assumptions. We apply this lower bound to the case of corrupted, bounded and semi-bounded distributions.
\end{itemize}
\vspace{-1em}
\section{Testing and estimation lower bounds}
\label{sec:testing_lb}


In order to obtain a lower bound on the performance of an estimator, we reduce that problem to testing, as is often done in the estimation literature (see for example \citep{tsybakov2009introduction}). Suppose that all distributions in $\mathcal D$ have finite mean. Then for any $x : \mathcal D \to \mathbb{R}$,
\begin{align*}
&\sup_{F \in \mathcal D} F\left(\vert \tilde{\mu}_n - \mu_F \vert \ge x(F)\right)
\\&= \sup_{c \in \mathbb{R}}\max\left\{\sup_{F \in \mathcal D, \mathbb{E}_{F}[X] \le c - x(F)} F(\tilde{\mu}_n \ge c),
  \sup_{G \in \mathcal D, \mathbb{E}_{G}[X] \ge c + x(G)} G(\tilde{\mu}_n \le c) \right\}
\: .
\end{align*}
For $E_c = \{\tilde{\mu}_n \ge c\}$ and $\mathcal F_c = \{F \in \mathcal D, \mathbb{E}_{F}[X] \le c - x(F)\}$, $\mathcal G_c = \{G \in \mathcal D, \mathbb{E}_{G}[X] \ge c + x(G)\}$, we obtain a lower bound of the form
\begin{align*}
\sup_{F \in \mathcal D} F\left(\vert \tilde{\mu}_n - \mu_F \vert \ge x(F)\right)
&\ge \sup_{c \in \mathbb{R}} \max\{\sup_{F \in \mathcal F_c}F(E_c), \sup_{G \in \mathcal G_c}G(E_c^c)\} \: .
\end{align*}
The only reason this is not an equality is that $E_c^c = \{\tilde{\mu}_n < c\}$ might differ from $\{\tilde{\mu}_n \le c\}$, and since we don't assume anything on $\tilde{\mu}_n$ the probability of the two events might be different.
In order to get estimation lower bounds, it then suffices to get lower bounds on $\max\{\sup_{F \in \mathcal F}F(E), \sup_{G \in \mathcal G}G(E^c)\}$ for an event $E$ and two families of distributions $\mathcal F$ and $\mathcal G$.
This can be seen as a testing bound: for a test trying to tell whether a distribution belongs to $\mathcal F$ or $\mathcal G$ and for $E$ the event that the test returns $\mathcal G$, this quantity is the error probability of the test.

\subsection{Divergences between probability distributions}

The lower bounds we prove for testing problems use tools from information theory, notably several divergences between probability distributions which we recall here.

\begin{definition}\label{def:kl_renyi}
Let $P$ and $Q$ be two probability distributions on a measurable space $\Omega$, let $\mu$ be a dominating measure of both (for example $P+Q$) and let $p = \frac{dP}{d\mu}$ and $q = \frac{dQ}{d\mu}$. We recall the definition of the following divergences, for $\alpha \in (0,1)$,
\begin{align*}
\text{Kullback-Leibler: }\KL(P,Q) &= \mathbb{E}_{X \sim P}\left[ \log \frac{d P}{d Q}(X) \right] \text{ if } P \ll Q \: , \KL(P, Q) = + \infty \text{ otherwise,}
\\
\text{Rényi: }D_\alpha(P,Q) &= - \frac{1}{1 - \alpha}\log \int p^\alpha q^{1-\alpha} \mathrm{d}\mu \: .
\end{align*}
\end{definition}
The Rényi divergence is finite if and only if $P$ and $Q$ are not mutually singular.
We refer the reader to \citep{van2014renyi} for a thorough review of the properties of the Rényi divergences and their relations to $\KL$.
Particularly relevant here are the relations $\frac{1}{2} D_{1/2}(P,Q) = -\log(1 - \H^2(P, Q))$, where $\H(P, Q) = \sqrt{\frac{1}{2}\int (\sqrt{p} - \sqrt{q})^2 \mathrm{d}\mu}$ is the Hellinger distance and $(1 - \alpha)D_\alpha(P,Q) = \inf_{R \in \mathcal P(\Omega)}(\alpha \KL(R, P) + (1 - \alpha)\KL(R, Q))$.
The Kullback-Leibler and Rényi divergences satisfy data processing inequalities, from which we can obtain the following lemma (proved in Appendix~\ref{app:testing_lb}), which is central to the proofs in this paper.

\begin{lemma}\label{lem:data_processing}
For all probability measures $P, Q$ on a measurable space $\Omega$, event $E$ and $\alpha \in (0,1)$,
\begin{align}
\KL(P, Q)
&\ge P(E) \log\frac{1}{Q(E)} - \log 2
\: ,\label{eq:kl_data_processing}
\\
(1 - \alpha) D_\alpha(P, Q)
&\ge \min\{\alpha, 1 - \alpha\} \log \frac{1}{\max\{P(E), Q(E^c)\}} - \log 2
\: .\label{eq:renyi_data_processing}
\end{align}
\end{lemma}

For KL, we will also use the following change of measure lemma, proved in Appendix~\ref{app:testing_lb}.
\begin{lemma}\label{lem:change_of_measure}
For all probability measures $P, Q \in \mathcal P(\Omega)$, all events $E$ of $\Omega$ and all $\beta>0$, 
\begin{align}
Q(E)e^{\KL(P, Q) + \beta}
&\ge P(E) - P\left(\log\frac{d P}{d Q} - \KL(P, Q) > \beta\right)
\: .\label{eq:kl_change_measure}
\end{align}
For all $n \in \mathbb{N}$, probability measures $P, F, G \in \mathcal P(\Omega)$, all events $E$ of $\Omega^n$ and all $\beta>0$,
\begin{align}
&2\max\{F(E), G(E^c)\}e^{n\max\{\KL(P, F), \KL(P, G)) + n\beta} \label{eq:kl_change_measure_3_points}
\\
&\ge 1 - P^{\otimes n}\left(\frac{1}{n}\sum_{i=1}^n\log\frac{d P}{d F}(X_i) - \KL(P, F) > \beta\right)
  - P^{\otimes n}\left(\frac{1}{n}\sum_{i=1}^n\log\frac{d P}{d G}(X_i) - \KL(P, G) > \beta\right)
\: .\nonumber
\end{align}
\end{lemma}

\begin{definition}
The Chernoff divergence between two finite measures $F$ and $G$ on a measurable space $\Omega$, denoted by $D_{\mathcal C}(F, G)$, is the quantity
$D_{\mathcal C}(F, G) = \inf_{P \in \mathcal P(\Omega)}\max\{\KL(P, F), \KL(P, G)\}$ .
\end{definition}
This is not a distance because it does not satisfy the triangle inequality, but it is symmetric, non-negative and equal to zero if $F = G$. Remark that it satisfies the inequality $D_{\mathcal C}(F, G) \ge \sup_{\alpha \in (0,1)} (1 - \alpha) D_\alpha(F, G)$. We will use in particular $D_{1/2}(F, G) \le 2 D_{\mathcal C}(F, G)$. For any divergence $D$ between probability distributions and any sets $\mathcal F, \mathcal G \subseteq \mathcal P(\Omega)$, we write $D(\mathcal F, \mathcal G) = \inf_{F \in \mathcal F, G \in \mathcal G} D(F, G)$.


\subsection{Lower bounds}
\label{sub:testing_bound}

From the lower bounds on the divergences discussed in the previous section, we can deduce lower bounds on the probability of error of any test tasked with deciding whether a distribution belongs to a set $\mathcal F$ or another set $\mathcal G$. 
\begin{theorem}\label{thm:non_asymptotic_testing}
Let $\mathcal F$ and $\mathcal G$ be two sets of probability distributions and let $E$ be an event. Then
\begin{align*}
\log\frac{1}{4\max\{\sup_{F \in \mathcal F} F(E), \sup_{G \in \mathcal G} G(E^c)\}}
\le D_{1/2}(\mathcal F, \mathcal G)
\le 2 D_{\mathcal C}(\mathcal F, \mathcal G) \: .
\end{align*}
\end{theorem}

\begin{proof}
Let $\delta = \max\{\sup_{F \in \mathcal F} F(E), \sup_{G \in \mathcal G} G(E^c)\}$. By inequality~\eqref{eq:renyi_data_processing}, for any $F \in \mathcal F$ and $G \in \mathcal G$, $D_{1/2}(F, G) \ge \log \frac{1}{4\delta}$. We conclude the proof by taking an infimum over $F$ and $G$.
\end{proof}

The two inequalities of the theorem are tight. Indeed, with $B_a$ the Bernoulli distribution with parameter $a$, let $\mathcal F= \{B_a\}$ and $\mathcal G = \{B_{1-a}\}$. Let $E = \{X = 1\}$. Then
\begin{align*}
\log\frac{1}{4\max\{B_a(E), B_{1-a}(E^c)\}}
&= \log\frac{1}{4 a}
\: , \\
D_{1/2}(B_a, B_{1-a})
= 2 D_{\mathcal C}(B_a, B_{1-a})
= 2 \KL(B_{\frac{1}{2}}, B_a)
&= \log\frac{1}{4a} + \log \frac{1}{1 - a}
\: .
\end{align*}
As $a \to 0$, these two quantities become arbitrarily close.
Theorem~\ref{thm:non_asymptotic_testing} is thus tight on Bernoulli distributions, hence on tests from one sample. However, if we consider two sets of product distributions on product spaces like $\mathbb{R}^n$, then as $n \to +\infty$ we get an inequality involving $D_{\mathcal C}$ instead of $2D_{\mathcal C}$ (see Theorem~\ref{thm:asymptotic_testing} in Appendix~\ref{app:testing_lb}).
The main tools used to remove the factor 2 are Lemma~\ref{lem:change_of_measure}, Equation~\eqref{eq:kl_change_measure_3_points}, and the law of large number to argue that the probabilities in Equation~\eqref{eq:kl_change_measure_3_points} go to zero.

We now combine the reduction from estimation to testing and the testing bound of Theorem~\ref{thm:non_asymptotic_testing}, to obtain an estimation lower bound.

\begin{theorem}\label{thm:generic_delta_bound}
For any $\left(x(n, \delta, F), \delta\right)$-estimator on a set $\mathcal D$ of distributions with finite means,
\begin{align}\label{eq:generic_delta_bound}
\frac{1}{n}\log \frac{1}{4\delta}
\le \inf_{c \in \mathbb{R}}
  D_{ 1/2}\left(\{F \in \mathcal D \mid \mathbb{E}_F[X] \le c - x(n, \delta, F)\}, \{G \in \mathcal D \mid \mathbb{E}_G[X] \ge c + x(n, \delta, G)\}\right)
\: .
\end{align}
\end{theorem}

We can use Theorem~\ref{thm:generic_delta_bound} to obtain lower bounds on $x(n, \delta, F)$ for various classes of distributions by obtaining upper bounds of the Rényi divergence between these two sets, or on the Chernoff divergence (since $D_{1/2} \le 2 D_{\mathcal C}$).
We now discuss two simple examples.

\begin{lemma}\label{lem:examples}
If $\mathcal D$ is the set of all Gaussian distributions and if there exists a $(\sqrt{2\Var_F[X]y(n,\delta)}, \delta)$-estimator. Then Equation~\eqref{eq:generic_delta_bound} simplifies to  $ \frac{1}{n}\log(1/4\delta)\le 2y(n ,\delta)$, and it follows that $y(n,\delta)\ge \frac{1}{2n}\log(1/4\delta)$.

\noindent If $\mathcal D$ is the set of all Laplace distributions and if there exists a $(\sqrt{2\Var_F[X] \log(1/4\delta)}, \delta)$-estimator. Then Equation~\eqref{eq:generic_delta_bound} simplifies to $\frac{1}{n}\log(1/4\delta)\le e^{-2\sqrt{y(n,\delta)}}+2\sqrt{y(n,\delta)}-1$ and it follows that $\delta \ge \frac{1}{4}e^{-4n}$.
\end{lemma}
The proof of this lemma is given in Appendix~\ref{sec:proof_lem_example}. Note that we get the same lower bounds for any set $\mathcal D$ that contains all Gaussian (resp. Laplace) distributions (but may be larger).

In Lemma~\ref{lem:examples}, in the Gaussian case, there is a factor $2$ in the denominator which seems to contrast with what we expect from the central limit theorem. This factor $2$ should indeed disappear asymptotically, this can be done with the same techniques as in the proof of Theorem~\ref{thm:lower_bound_finite_var}.
However, as already discussed below Theorem~\ref{thm:non_asymptotic_testing}, when $n=1$ this factor $2$ cannot be removed in general if $\mathcal{D}$ contains Bernoulli distributions. In the case of Laplace distributions, we get that an estimator cannot be sub-Gaussian with optimal constant unless $\delta \ge \frac{1}{4}e^{-4n}$, this result was already proved with a slightly worst constant in \cite[Theorem 4.3]{devroye2016sub} (their lower bound is $\delta \ge e^{1-10n}$).

\section{Lower bounds for nonparametric estimation}

\subsection{The set of distributions with finite variance}
\label{sub:the_set_of_distributions_with_finite_variance}

We consider in this section the set of distributions $\mathcal D_2 = \{F \in \mathcal{P}(\mathbb{R}) \mid \mathbb{E}_{X \sim F}[X^2] < + \infty\}$. We are interested in lower bounds on $\left(\sqrt{2\Var_F[X] y(n, \delta)}, \delta\right)$-estimators, which include the sub-Gaussian estimators for $y(n, \delta) = \log (2/\delta)/n$.
In order to apply Theorem~\ref{thm:generic_delta_bound}, it suffices to be able to compute a $1/2$-Rényi or Chernoff divergence between two particular sets of distributions.
This is what we do in the next lemma, which uses heavily the work done in \citep{nishiyama2020tight} on the computation of the minimal Hellinger distance between two sets of distributions with given mean and variance.

\begin{lemma}\label{lem:renyi_div_variance}
For all $c \in \mathbb{R}$ and $x \ge 0$,
\begin{align*}
&D_{1/2}\left(\left\{F \in \mathcal D_2 \mid \mathbb{E}_F[X] \le c - \sqrt{2\Var_F[X] x}\right\},
  \left\{G \in \mathcal D_2 \mid \mathbb{E}_G[X] \ge c + \sqrt{2\Var_G[X] x}\right\}\right)
\\
&= \log (1 + 2 x)
\: .
\end{align*}
Furthermore, for these pairs of sets, $D_{1/2} = 2 D_{\mathcal C}$.
\end{lemma}

The proof can be found in Appendix~\ref{app:finite_variance}.
Now that we have computed the Chernoff divergence between the two sets of measures used in the testing bound, we get an estimation lower bound.

\begin{theorem}\label{thm:lower_bound_finite_var}
For any $\left(\sqrt{2\Var_F[X] y(n, \delta)}, \delta\right)$-estimator on $\mathcal D_2$, $y(n, \delta)$ satisfies the inequalities
\begin{align*}
\frac{1}{n} \log \frac{1}{4\delta} &\le \log (1 + 2 y(n, \delta))
\: , \\
\frac{1}{n} \log \frac{1}{4\delta} &\le \frac{1}{2} \log (1 + 2 y(n, \delta)) + \frac{2}{\sqrt{n}} \log(\sqrt{1 + 2y(n, \delta)} + \sqrt{2 y(n, \delta)})
\: .
\end{align*}
If the error probability depends on $n$ with $\delta_n \to 0$, then $\frac{2}{\sqrt{n}} \log(\sqrt{1 + 2y(n, \delta_n)} + \sqrt{2 y(n, \delta_n)}) = o\left( \frac{1}{2} \log (1 + 2 y(n, \delta_n)) \right)$.
\end{theorem}

The first inequality is an immediate consequence of Theorem~\ref{thm:generic_delta_bound} and Lemma~\ref{lem:renyi_div_variance}. To obtain the second inequality, we use Lemma~\ref{lem:change_of_measure}, Equation~\eqref{eq:kl_change_measure_3_points}, with explicit minimizers $F_n$ and $G_n$ for the Rényi divergence (coming from \citep{nishiyama2020tight}).
Given these minimizers, we are able to control the variance of the likelihood ratio appearing in Equation~\eqref{eq:kl_change_measure_3_points}. The two expressions on the right of the second inequality then correspond to the $\KL$ and variance terms.
While Theorem~\ref{thm:generic_delta_bound} is more easily applicable to many problems, the finer knowledge we have of the minimizers of the Rényi divergence for the finite variance case allows us to get a better asymptotic bound with Equation~\eqref{eq:kl_change_measure_3_points}.

\paragraph{A new estimator}
\label{par:a_new_estimator}

We introduce a new mean estimator, based on our lower bounds, and prove that its error is close to the lower bound in the regime $\delta \ll 1/n$.
The estimator is however hard to compute, and its main purpose is to be a witness of the tightness of the lower bounds.

For $c \in \mathbb{R}$ and a positive function $y(n,\delta)$, let $\mathcal F_c = \{F \in \mathcal D_2 \mid \mathbb{E}_F[X] \le c - \sqrt{2 \Var_F[X] y(n, \delta)}\}$ and $\mathcal G_c = \{G \in \mathcal D_2 \mid \mathbb{E}_G[X] \ge c + \sqrt{2 \Var_G[X] y(n, \delta)}\}$. Let $\hat{P}_n$ be the empirical distribution of $n$ samples and let $\hat{d}_{L, n} : c \mapsto \inf_{F \in \mathcal F_c}\KL(\hat{P}_n, F)$ and $\hat{d}_{R, n} : c \mapsto \inf_{G \in \mathcal G_c}\KL(\hat{P}_n, G)$.
The function $\hat{d}_{L, n}$ is non-negative, non-increasing and equal to 0 for $c$ large enough.
The function $\hat{d}_{R, n}$ is non-negative, non-decreasing and equal to 0 for $c$ small enough.
We define the estimator
\begin{align*}
\tilde{\mu}_n = \sup\{c \in \mathbb{R} \mid \hat{d}_{L, n}(c) > \hat{d}_{R, n}(c) \} \: .
\end{align*}

\begin{theorem}\label{thm:estimator_upper_bound}
For $y(n, \delta) = \frac{1}{2}\left( \exp \left( \frac{2}{n}\log\frac{2e(n+1)^2}{\delta} \right) - 1 \right)$, $\tilde{\mu}_n$ is a $(\sqrt{2 \sigma^2 y(n, \delta)}, \delta)$-estimator.
\end{theorem}
The idea of the proof is that on one hand, $\max\{\hat{d}_{L, n}(c), \hat{d}_{R, n}(c)\}$ is larger than a Chernoff divergence. On the other hand, if the estimator $\tilde{\mu}_n$ makes a large mistake, then with probability $1 - \delta$ the same quantity has to be small by a concentration argument about the infimum of Kullback-Leibler divergence, and we obtain a contradiction if $y(n,\delta)$ is well-chosen.
\begin{lemma}\label{lem:Kinf_concentration}
Let $\hat{P}_n$ be the empirical distribution of $n$ samples from a probability distribution $P$. If $\mathbb{E}_P[X] \le m $ and $ \Var_P[X] \le \sigma^2$ then with probability $1 - \delta$, for all $n \in \mathbb{N}$,
\begin{align*}
\inf_{F, \mathbb{E}_F[X] \le m, \Var_F[X] \le \sigma^2}\KL(\hat{P}_n, F)
< \frac{1}{n}\log \frac{e(n+1)^2}{\delta}
\: .
\end{align*}
The same result holds with all inequalities $\mathbb{E}_P[X] \le m$ replaced by $\mathbb{E}_P[X] \ge m$.
\end{lemma}
See Appendix~\ref{app:finite_variance} for a proof, which follows previous work \citep{agrawal2021optimal} and uses duality of the Kullback-Leibler divergence and a bound on the regret of an online learning algorithm.

From Theorem~\ref{thm:lower_bound_finite_var}, we get that for any $(\sqrt{2 \sigma^2 y(n, \delta)})$-estimator, $y(n, \delta)$ has to be larger than $\frac{1}{2}\left( \exp \left( \frac{1}{n}\log\frac{1}{4\delta} \right) - 1 \right)$. And if $\delta \to 0$, the lower bounds become closer to $\frac{1}{2}\left( \exp \left( \frac{2}{n}\log\frac{1}{4\delta} \right) - 1 \right)$.
The difference between this second lower bound and the error width of $\tilde{\mu}_n$ is the term $2e(n+1)^2$ in the logarithm.
The estimator $\tilde{\mu}_n$ is thus close to optimal in the regime $\delta \ll 1/n$, but for larger $\delta$ its error bound depends on $\log(n)/n$ instead of $\log(1/\delta)/n$.
The $\log(n)$ term comes from the KL concentration lemma~\ref{lem:Kinf_concentration}, which is a concentration bound valid for all sample sizes: a $\log(n)$ factor is the usual price for obtaining a bound valid for all times. We would like to have a concentration result for a given $n$ without that $\log(n)$ term, and we conjecture that such a concentration inequality should hold (in which case the estimator is optimal for all $\delta$), but this remains an open question.

Together, Theorem~\ref{thm:lower_bound_finite_var} and Theorem~\ref{thm:estimator_upper_bound} characterize the optimal error for an estimator with error probability $\delta$ when $\delta \ll 1/n$. We see that the notion of sub-Gaussian estimator is inadequate when $\frac{1}{2}\left( \exp \left( \frac{2}{n}\log\frac{1}{4\delta} \right) - 1 \right)$ is not within a constant of $\frac{1}{n}\log\frac{1}{\delta}$, but is much larger.

\subsection{Lower bound for finite $1+\alpha$ moments}
For $\alpha \in (0,1)$, let $\mathcal D_{1+\alpha} = \{F \in \mathcal P(\mathbb{R}) \mid \mathbb{E}_{F}[|X-\E_F[X]|^{1+\alpha}] < + \infty\}$ be the distributions with finite moment of order $1 + \alpha$. For $F \in \mathcal D_{1+\alpha}$, we write $M_{1+\alpha}(F)=\E_F[|X-\E_F[X]|^{1+\alpha}]^{1/(1+\alpha)}$.
\vspace{-0.75em}

\begin{theorem}\label{thm:lower_bound_finite_1_alpha}
Let $\delta \in (0,1/2)$. For any $\left(\varepsilon_n M_{1+\alpha}(F), \delta\right)$-estimator on $\mathcal D_{1+\alpha}$, with $\alpha \in (0,1)$, $\varepsilon_n$ satisfies
\vspace{-0.75em}
\begin{align*}
\varepsilon_n \ge \frac{1}{2^{\frac{1}{1 + \alpha}}} \left(\frac{1}{n}\log\frac{1}{2\delta}\right)^{\frac{\alpha}{1+\alpha}}
\: .
\end{align*}
If $\log(1/\delta)/n \to 0$, then the bound becomes $\varepsilon_n \ge \frac{1}{1 + o(1)} \left(\frac{1}{n}\log\frac{1}{2\delta}\right)^{\frac{\alpha}{1+\alpha}}$.
\end{theorem}
\cite[Theorem 3.1]{devroye2016sub} is a very similar lower bound on that problem, and we use the distributions they build for their lower bounds in our proof. The lower bound technique is however different. There is a small mistake in the constants in their proof, and their corrected result is that for $\delta \in (e^{-4n}, 1/2)$, $\varepsilon_n \ge 2^{-\frac{1 - \alpha}{1 + \alpha}} \left(\frac{1}{n}\log\frac{2}{\delta}\right)^{\frac{\alpha}{1+\alpha}}$.
The differences between their result and ours are first that we have a worse power of 2 in front and second, that they have a $\log(2/\delta)$ where we have $\log(1/2\delta)$, so their bound is slightly better. Note however that we don't need $\delta > e^{-4n}$.

\subsection{Non-asymptotic lower bound under log-Lipschitz condition}

\begin{theorem}[Bound on KL between shifted log-Lipschitz distributions]\label{thm:non_asymptotic_lip} 
For any fixed $\delta>0$, let $\widehat{\mu}$ be a $\left(\varepsilon_n, \delta\right)$-estimator for all translations\footnote{We call translation of a distribution $F$ by $h$ the distribution with cumulative distribution function $F(\cdot - h)$.} of a distribution $F$. Let $f$ be the probability distribution of $F$ and suppose that $\log(f)$ is $L$-Lipschitz. Then, we have
\begin{align*}
\varepsilon_n \ge\frac{1}{L}\log\left(1+\sqrt{4\left(1-e^{-\frac{\log(1/4\delta)}{2n}}\right)}\right).
\end{align*}

\end{theorem}
Remark that when $\frac{\log(1/4\delta)}{n} \to 0$, we get the simpler lower bound 
$\varepsilon_n \ge\frac{1}{L}\sqrt{\frac{2\log(1/4\delta)}{n}}.$
The proof uses Theorem~\ref{thm:generic_delta_bound}, together with an upper bound on divergences of the form
$\KL(F(x), F(x+h))$. See Appendix~\ref{app:finite_variance}. We now detail two examples of application of this theorem.

\noindent\textbf{Laplace distributions.} If $f$ is a Laplace distribution $f(x)=\frac{1}{\sigma^2}\exp(-|x-\mu|\sqrt{2}/\sigma)$ with variance $\sigma^2$, then $L = \sqrt{2}/\sigma$ and $\varepsilon_n \ge\sigma\sqrt{\frac{\log(1/4\delta)}{n}}$ when $\frac{\log(1/4\delta)}{n} \to 0$. Remark that alternatively, we can use Theorem~\ref{thm:without_2} to get a similar result in an asymptotic regime (up to a $\sqrt{2}$ factor).\\
\textbf{Huber corruption neighborhood of the Gaussian.}
If $\mathcal{D}= \{ F \in \mathcal{P}(\R) | F=(1-\varepsilon)\Phi + \varepsilon H, \, H \in \mathcal{P}(\R) \}$ is Huber's Gaussian corruption neighborhood for some $\varepsilon\in (0,1/2)$. Then, using the Lipshitz constant from Equation (4.51) in \citep{robuststat} for  the corruption $\varepsilon=0.2$, and when $\frac{\log(1/4\delta)}{n} \to 0$, the bound becomes $\varepsilon_n \ge\frac{1}{0.87}\sqrt{\frac{\log(1/4\delta)}{n}}$. The constant is not as good as in the asymptotic bound obtained later (see discussion after Theorem~\ref{thm:local_lowerbound_fisher}), it is smaller by a factor $2.58$.

\section{Lower bound on estimation in (semi-)parametric families}\label{sec:parametric}

In this section, we apply our lower bound technique to the case of the location model $\mathcal{D}=\{F(x-\mu) \mid \mu \in \R\}$ for some smooth distribution $F$ and to smooth bounded or semi-bounded distributions. 

\subsection{Quadratic Mean Differentiability and the Fisher information for parametric lower bounds}

We assume that we are dealing with a parametric family of distributions that are smooth in the sense that the distributions form a quadratic mean differentiable family (abbreviated q.m.d.) defined in Definition~\ref{def:qmd}.
See Section 7.2 in \citep{le2000asymptotics} for background on this concept.

\begin{definition}[Q.m.d family of distributions]\label{def:qmd}
Let $\Theta \subset \R$, and let $\{P_\theta  \mid  \theta \in \Theta\}$ be a family of distributions absolutely continuous with respect to a measure $\mu$.  $\{P_\theta \mid \theta \in \Theta\}$ is said to be quadratic mean differentiable  at $\theta_0\in \Theta$ if there exists $\eta:\R\times \Theta \to \R$ such that for $|h| \to 0$
$$\int_\R \left(\sqrt{p_{\theta_0+h}(x)}-\sqrt{p_{\theta_0}} - h \eta(x,\theta_0)\right)^2\d \mu x = o(|h|^2)\: .$$ 
\end{definition}

In order to prove that a family of distributions is q.m.d., one may use \citep[Theorem 12.2.1]{lehmann2005testing} that we recall in Appendix~\ref{sec:app_qmd}. It can be shown that most exponential families of distributions are q.m.d. Another example is Laplace distributions: although it is not smooth, the Laplace family of distributions parametrized by its mean is q.m.d. (see Section 7.2 of \citep{RePEc:cup:cbooks:9780521784504}). The family of uniform distributions is an example of distributions which are not q.m.d.

Let $\mathcal{F}=\{F_\theta \mid \theta \in \R\}$ be a q.m.d. family of distributions on $\R$ absolutely continuous with respect to a common measure $\mu$. We denote $\I(F_\theta)$ the Fisher information defined by $\I(F_\theta):=4\int\eta(x,\theta)^2\d \mu(x)$ .
If moreover the distributions are continuously differentiable with respect to the parameter, we have $\I(F_\theta) := \E_{F_\theta}\left[(\frac{\partial f_\theta}{\partial\theta}(X))^2/f_\theta(X)^2\right].$
In particular, we will be interested in the parametric location family defined for some distribution $F$ by  $\mathcal{F}=\{F(x-\mu) \mid \mu \in \R\}$, for which we will assume that $f$ is derivable and the Fisher information becomes
$$\I(F) := \E_{F}\left[\frac{(f'(X))^2}{f(X)^2}\right].$$

The following theorem found  in \citep[Theorem 12.2.3]{lehmann2005testing} uses the Fisher information to specify the asymptotic mean and variance of the likelihood ratios in a q.m.d. family.

\begin{theorem}[Convergence of likelihood ratio for a q.m.d. family]\label{thm:clt_qmd}
Let $\Theta \subset \R$, suppose that $\{P_\theta \mid  \theta \in \Theta\}$ is q.m.d. and absolutely continuous with respect to a measure $\mu$ on $\R$. Suppose that $0<\I(P_\theta)<\infty$ and fix $\theta_0\in \Theta$. We have under $P_{\theta_0}^{\otimes n}$, for any $h \in \R$, $h\neq 0$,
$$\sum_{i=1}^n \log\left(\frac{p_{\theta_0}}{p_{\theta_0+h/\sqrt{n}}}\right)\xrightarrow[n \to \infty]{d}\mathcal{N}\left(\frac{h^2}{2}\I(P_{\theta_0}),h^2\I(P_{\theta_0})\right)  $$
\end{theorem}

It is straightforward to see that the mean of the left-hand-side of the limit in Theorem~\ref{thm:clt_qmd} is $n \KL(P_{\theta_0},P_{\theta_{0}+h/\sqrt{n}})$. The expression for the mean in the asymptotics follows because the Fisher information is linked to the second derivative of KL (\cite[Theorem 4.4.5]{calin2014geometric} and \cite[Theorem 2.1]{abbasnezhad2006relations}). We now state the main result of this section.

\begin{theorem}\label{thm:without_2}
Let $\Theta \subset \R$ and let $\{P_\theta \mid \theta \in \Theta\}$ be a q.m.d. family which is absolutely continuous with respect to a measure $\mu$ on $\R$. Let $\uI = \inf \{\I(P_\theta) \mid  \theta \in \Theta, \: 0<\I(P_\theta)<\infty\}$.
If there exists a sequence of $(\eta/\sqrt{n}, \delta)$-estimators on $\{P_\theta \mid \theta \in \Theta\}$ for samples sizes $n \to +\infty$, then $\eta$ satisfies
\begin{align*}
\eta \ge \frac{1}{\sqrt{\uI}}\left(\sqrt{1 + 2 \log\frac{1}{4 \delta}} - 1\right)
\: .
\end{align*}
If there exists a family of estimators $(\mathcal E_{n,\delta})_{n \in \mathbb{N}, \delta \in (0,1)}$ and reals $(\eta_\delta)_{\delta \in (0,1)}$ such that each $\mathcal E_{n, \delta}$ is a $(\eta_\delta/\sqrt{n}, \delta)$-estimator, then
$
\liminf_{\delta \to 0} \eta_\delta \left( \sqrt{\frac{2}{\uI}\log\frac{1}{\delta}} \right)^{-1} \ge 1
$ .
\end{theorem}

Theorem~\ref{thm:without_2} is more general than \cite[Theorem 1.3]{gupta2022finite} because they compare smoothed Fisher informations while our result is applicable directly to the target distribution.
The proof is an application of Lemma~\ref{lem:change_of_measure}, equation~\eqref{eq:kl_change_measure_3_points}, in which we control the two probabilities of deviation with Theorem~\ref{thm:clt_qmd} (see Appendix~\ref{app:parametric}).

\subsection{Lower bounds on location estimation based on the Fisher information}
\label{sub:fisher_info}

\subsubsection{Distribution that are invariant by translation}\label{sec:bound_fisher_R}

\begin{corollary}\label{cor:local_lowerbound_fisher_location}
Let $F\in\mathcal D$, with $\mathcal{D}$ invariant by translation (i.e. $F(\cdot +h) \in \mathcal{D}$ for any $h \in \R$), $F$ absolutely continuous with respect to a measure $\mu$. Let $f= \frac{\d F}{\d \mu}$ be the density of $F$, suppose that $f$ is absolutely continuous and that $\I(F)<\infty$. If there exists a sequence of $(\eta/\sqrt{n}, \delta)$-estimators on $\{F(\cdot +h)\mid  h \in \R\}$ for samples sizes $n \to +\infty$, then $\eta$ satisfies
\begin{align*}
\eta \ge \frac{1}{\sqrt{\uI}}\left(\sqrt{1 + 2 \log\frac{1}{4 \delta}} - 1\right)
\: .
\end{align*}
\end{corollary}
Corollary~\ref{cor:local_lowerbound_fisher_location} is a lower bound valid for any location estimator, not only for estimation of the mean, as is the bound of Theorem~\ref{thm:local_lowerbound_fisher}. In particular, in the case of robust mean estimation in a corrupted setting, the lower bound will ignore the asymptotic bias that is inevitable when dealing with corrupted distributions, see \cite[Section 4.2]{robuststat}. 

Until now, in this section, everything was in a fully parametric setting, using a smooth parametric family of distribution. However, we can use Corollary~\ref{cor:local_lowerbound_fisher_location} in a semi-parametric way. Let us define $\mathcal{D}_F = \{F(\cdot -h) \mid  h \in \R\}$ and $\mathcal{D}=\cup_F \mathcal{D}_F$.
We can obtain a lower bound for each $\mathcal D_F$ with Corollary~\ref{cor:local_lowerbound_fisher_location} and then combine these lower bounds to obtain that if there exists a sequence of $(\eta/\sqrt{n}, \delta)$-estimators on $\mathcal{D}$ for samples sizes $n \to +\infty$, then $\eta$ satisfies
$$\eta \ge \frac{1}{\inf_{F \in \mathcal{D}}\I(F)}\left(\sqrt{1 + 2 \log\frac{1}{4 \delta}} - 1\right).
$$
From this, we see that it may be interesting to look at minimizer of the Fisher information over nonparametric set of distributions and our lower bound would still apply to those.

\begin{lemma}\label{lem:minmizers_fisher}
Let $\mathcal{D}$ be a set of distribution and denote $F_{\mathcal{D}} = \arg\min_{F \in \mathcal D} \I(F)$.
\begin{itemize}
\item If $\mathcal{D}= \{F \in \mathcal{P}(\R) \mid \int (x^2-1)F(\d x) \le 0\}$ is the set of distributions with second moment bounded by $1$, then $F_{\mathcal{D}}=\Phi$ is the standard normal distribution for which $\I(\Phi) = 1$.
\item If $\mathcal{D}= \left\{F \in \mathcal{P}(\R)  \mid  F\{(-1,1)\} \ge 1-\varepsilon\right\}$ for $\varepsilon\in (0,1)$. Let $\omega\in(0,\pi)$ be such that $\varepsilon = 2\cos^2(\omega/2)/(\omega \tan(\omega/2)+2).$
Then, 
$I_1(\varepsilon):=\I(F_{\mathcal{D}})=\frac{\omega^2}{1+2/(\omega\tan(\omega/2))}.$
\item If $\mathcal{D}= \{ F \in \mathcal{P}(\R)  \mid  F=(1-\varepsilon)\Phi + \varepsilon H, \, H \in \mathcal{P}(\R) \}$ is Huber's Gaussian corruption neighborhood for some $\varepsilon\in (0,1/2)$. Let $k>0$ be related to $\varepsilon$ by $2\varphi(k)/k-2\Phi(-k)=\varepsilon/(1-\varepsilon)$ then 
$ 
I_2(\varepsilon):=\I(F_{\mathcal{D}})=(1-\varepsilon)\left(2\Phi\left(\frac{k}{2}\right)-1\right) +\frac{1-\varepsilon}{\sqrt{2\pi}}\left(\frac{2 \, e^{\left(-\frac{1}{2} \, k^{2}\right)}}{k}- 2 k e^{\left(-\frac{1}{2} \, k^{2}\right)}\right).
$
\end{itemize}
\end{lemma}

\begin{figure}[htbp]
\floatconts
{fig:corruption_min_fisher}
{\caption{Minimal Fisher information over two set of corrupted distributions from Lemma~\ref{lem:minmizers_fisher}.}}
{%
\subfigure[plot of $I_1(\varepsilon)$]{%
\label{fig:fisher_I1}
\includegraphics[width=0.45\textwidth]{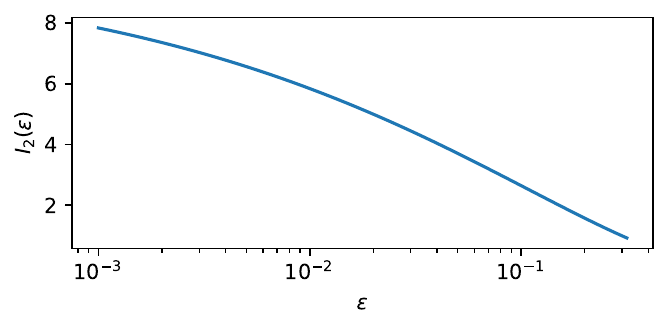}
} 
\subfigure[plot of $I_2(\varepsilon)$]{%
\label{fig:fisher_I2}
\includegraphics[width=0.45\textwidth]{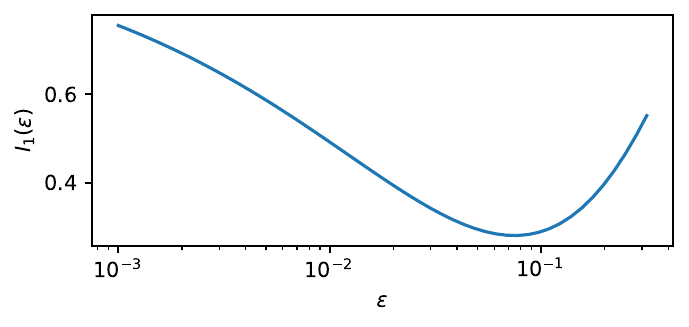}
}
}
\end{figure}

The proof of Lemma~\ref{lem:minmizers_fisher} can be found in \cite[p.81]{robuststat} in which can also be found the minimizer of the Fisher information over a neighborhood in Kolmogorov distance and for more general corruption neighborhoods. The plots of $I_1(\varepsilon)$ and $I_2(\varepsilon)$ can be found in Figure~\ref{fig:corruption_min_fisher}. In Figure~\ref{fig:fisher_I1}, we see that the Fisher information is always smaller than $1$ which means that the resulting estimation problem is harder than the non-corrupted case with finite variance. On the other hand, in Figure~\ref{fig:fisher_I2}, the Fisher information is almost always larger than $1$ which says that the information of how much mass is outside $(-1,1)$ makes the estimation problem easier.

A direct application of Lemma~\ref{lem:minmizers_fisher} together with Theorem~\ref{thm:local_lowerbound_fisher} gives lower bound for the estimation error. For example for $\varepsilon=0.2$, for any $\left(\eta/\sqrt{n}, \delta\right)$-estimator on $ \{ F \in \mathcal{P}(\R) \mid F=0.8\Phi + 0.2 H, \, H \in \mathcal{P}(\R) \}$ for $n \to \infty$, we have $\eta \ge \sqrt{\Var_F[X]\frac{2\log(1/4\delta)}{0.40}},$
and for any $\left(\eta/\sqrt{n}, \delta\right)$-estimator on $ \{ F \in \mathcal{P}(\R)  \mid F\{(-1,1)\} \ge 0.8 \}$ for $n \to \infty$, we have
$
\eta \ge \sqrt{\Var_F[X]\frac{2\log(1/4\delta)}{1.57}}.
$

There already exists several distribution sets for which the minimizer of the Fisher information is known. Such distributions can then be plugged into Theorem~\ref{thm:local_lowerbound_fisher} to provide a lower bound on the associated estimation problem. In particular, \citep{vzivojnovic1998minimum} looked at the case of moment constraints and \citep{bickel1983minimizing} looked at the case of mixture of smooth distributions and \citep{robuststat} looked at various corrupted neighborhood of parametric distributions. 

\subsubsection{Distribution with bounded or semi-bounded support}

In the case of distributions on $\R_+$ or supported on an interval, the minimal Fisher information is known (see \cite{bercher2009minimum}) but we cannot apply the results from Section~\ref{sec:bound_fisher_R} because $\KL\left(F, F(\cdot - h)\right)$ will be infinite as soon as $h > 0$ in the case of support on $\R_+$ and it will always be infinite for support on an interval. Instead, we consider the random variable $Y_h = X + 2 h U $ where $U\sim \mathrm{Unif}([0,1])$. This can be seen as a smoothing.
Remark that the c.d.f. or $Y_h$ is 
$$F_{Y_h}(x) = F\left(X + 2 h U\le x \right) = \frac{1}{h}\int_0^h F(x-2u)\d u $$
whose p.d.f is 
$f_{Y_h}(x) = \frac{1}{h}\int_0^h f(x-2u)\d u.$ Remark that because $f$ is absolutely continuous, $f_{Y_h}$ is derivable with respect to $h$. 
In order to have a well-defined Fisher information, we take the convention that for any $x$ such that $f(x)=0$, $f'(x)/f(x)=0$.



\begin{theorem}\label{thm:local_lowerbound_fisher}
If there exists a sequence of $(\eta/\sqrt{n}, \delta)$-estimators for samples sizes $n \to +\infty$ on the set of distributions $F$ supported in $[a, b]$ with density $f$ that is absolutely continuous on $[a,b]$ for some $a,b\in \R$, then $\eta$ satisfies for any such $F$
\begin{align*}
\eta \ge \frac{1}{\sqrt{\I(F)}}\left(\sqrt{1 + 2 \log\frac{1}{4 \delta}} - 1\right)
\: .
\end{align*}
\end{theorem}
Then, using the minimum Fisher information for bounded and semi-bounded distributions from \cite{bercher2009minimum}, we conclude in Corollary~\ref{cor:lower_bound_semi_bounded} and Corollary~\ref{cor:lower_bound_bounded} on lower bounds on mean estimation for these two families of distributions.

\begin{corollary}\label{cor:lower_bound_semi_bounded}
If there exists an $\left(\eta/\sqrt{n}, \delta\right)$-estimator on $F$, a distribution supported on $[0,\infty]$ with absolutely continuous density on $[0,\infty)$, then $\eta$ satisfies
$
 \eta \ge \sqrt{\frac{\Var_F[X]}{9-24/\pi}}\left(\sqrt{1+2\log\frac{1}{4\delta}}-1\right).
$
\end{corollary}

\begin{corollary}\label{cor:lower_bound_bounded}
For any $\left(\eta/\sqrt{n}, \delta\right)$-estimator on the set of distribution supported on $[a,b]$, $\eta$ satisfies
$
 \eta \ge \frac{b-a}{\pi}\left(\sqrt{1+2\log\frac{1}{4\delta}}-1\right).
$
\end{corollary}

\section{Conclusion and open problems}
\label{sec:conclusion_and_open_problem}

We have provided lower bounds for mean and location estimation and have shown several applications of the bounds, to nonparametric settings like the distributions with bounded variance as well as to parametric cases with lower bounds depending on the Fisher information.

About the generic lower bounds, we observed in Theorem~\ref{thm:non_asymptotic_testing} that the factor 2 in the $2 D_{\mathcal C}$ upper bound is necessary for Bernoulli distributions and $n=1$ sample. However, that factor 2 disappears when $n \to + \infty$, a fact that we prove with a different proof and that requires more information about the set of distributions we consider.
An interesting goal would be to have a single bound with a factor transitioning from 2 to 1 as $n$ grows (possibly under assumptions on the sets of distributions).

For the finite variance case, the main remaining questions are first whether the estimator we introduced is optimal also for $\delta > 1/n$, which could require a better concentration on the infimum of Kullback-Leibler divergence, and second whether we can compute efficiently that estimator (or modify it to make it efficiently computable).

In the case of the distributions with finite moment of order $1 + \alpha$, our lower bound could presumably be improved by computing exactly the Chernoff divergence between the sets of distributions with high and low mean and the optimal distributions as was done in the variance case. The current bound is based on a guess for good distributions taken from \citep{devroye2016sub}.

There have been a lot of interest in robust multivariate mean estimation, in particular lower bounds can be found in \cite[Theorem 2.2]{10.1214/17-AOS1607}, the authors show that the optimal rate is of order $\Omega(\sqrt{\mathrm{Tr}(\Sigma)/n}+\sqrt{\|\Sigma\|\log(1/\delta)/n})$ where $\Sigma$ is the covariance matrix. Several estimators achieve this optimal rate~\citep{depersin2019robust,10.1214/21-AOS2145} but the optimal constants in factor of the complexities term are still unknown. 

\acks{The authors acknowledge the funding of the French National Research Agency under the project FATE (ANR-22-CE23-0016-01).
This work benefited from the support of the French Ministry of Higher Education and Research, of Inria and of the Hauts-de-France region.
The authors are members of the Inria Scool team.}

\bibliography{subgaussian_bib}
\newpage

\appendix


\section{Proofs: testing and estimations lower bounds}
\label{app:testing_lb}

\begin{proof}[of Lemma~\ref{lem:data_processing}]
The first inequality is a simple consequence of the data processing inequality for $\KL$, and can be found for example in \citep[Lemma 5 and Equation 11]{garivier2019explore}.

Let's now prove the Rényi divergence inequality. First, use the equality $(1 - \alpha)D_\alpha(F, G) = \inf_{P\in \mathcal P(\Omega)}\left(\alpha \KL(P, F) + (1 - \alpha) \KL(P, G)\right)$ \citep[Theorem 27]{van2014renyi}. Let then $P \in \mathcal P(\Omega)$.
Let $\delta = \max\{F(E), G(E^c)\}$. By Equation~\eqref{eq:kl_data_processing} we have $\KL(P, F) \ge P(E) \log\frac{1}{F(E)} - \log 2$ and $\KL(P, G) \ge P(E^c) \log\frac{1}{G(E^c)} - \log 2$. Let $\alpha \in (0, 1)$. Summing the two inequalities with weights $\alpha$ and $(1 - \alpha)$ we get
\begin{align*}
\alpha \KL(P, F) + (1 - \alpha) \KL(P, G) \ge (\alpha P(E) + (1 - \alpha) P(E^c)) \log\frac{1}{\delta} - \log 2
\: .
\end{align*}
Since $P$ is arbitrary we can compare the infimum over $P$ of both sides.
\begin{align*}
\inf_{P\in \mathcal P(\Omega)}\left(\alpha \KL(P, F) + (1 - \alpha) \KL(P, G)\right)
&\ge \inf_{P \in \mathcal P(\Omega)}(\alpha P(E) + (1 - \alpha) P(E^c)) \log\frac{1}{\delta} - \log 2
\\
&= \min\{\alpha, 1 - \alpha\} \log \frac{1}{\delta} - \log 2
\: .
\end{align*}
\end{proof}

\begin{proof}[of Lemma~\ref{lem:change_of_measure}]
First inequality:
\begin{align*}
Q(E)
&= \E_P\left[\mathbb{I}(E) e^{- \log \frac{dP}{dQ}}\right]
\ge \E_P\left[\mathbb{I}\left(E \cap \left\{\log \frac{dP}{dQ} \le \KL(P, Q) + \beta\right\}\right) e^{- \log \frac{dP}{dQ}}\right]
\\
&\ge e^{- \KL(P, Q) - \beta} P\left(E \cap \left\{\log \frac{dP}{dQ} \le \KL(P, Q) + \beta\right\}\right)
\\
&\ge e^{- \KL(P, Q) - \beta} \left( P(E) - P\left(\log \frac{dP}{dQ} - \KL(P, Q) > \beta\right)\right)
\: .
\end{align*}
For the second inequality, we apply the first inequality once with $P^{\otimes n}$ and $F^{\otimes n}$ and event $E$, and once for $P^{\otimes n}$ and $G^{\otimes n}$ and event $E^c$. We then sum the two inequalities and use $P(E) + P(E^c) = 1$. We get
\begin{align*}
&F(E) e^{n \KL(P, F) + n\beta} + G(E^c) e^{n\KL(P, G) + n\beta}
\\
&\ge 1 - P^{\otimes n}\left(\frac{1}{n}\sum_{i=1}^n\log\frac{d P}{d F}(X_i) - \KL(P, F) > \beta\right)
  - P^{\otimes n}\left(\frac{1}{n}\sum_{i=1}^n\log\frac{d P}{d G}(X_i) - \KL(P, G) > \beta\right)
\: .
\end{align*}
We finally use that the sum of two terms on the left-hand side is less than twice the maximum.
\end{proof}

For a function $\phi: \mathcal X \to \mathcal Y$ between two measurable spaces and a measure $F$ on $\mathcal X$, we denote the pushforward of $F$ by $\phi$ by $\phi_*F$. This is a measure on $\mathcal Y$.
For a set of distributions $\mathcal F$ on $\mathcal X$, we write $\phi_* \mathcal F = \{\phi_* F \mid F \in \mathcal F\}$.

\begin{lemma}\label{lem:DC_measurable_equiv}
If $\phi : \mathcal X \to \mathcal Y$ is an equivalence between measurable spaces and $F, G$ are two distributions on $\mathcal X$, then $D_{\mathcal C}(\phi_* F, \phi_* G) = D_{\mathcal C}(F, G)$. If $\mathcal F$ and $\mathcal G$ are two sets of distributions on $\mathcal X$, we also have $D_{\mathcal C}(\phi_* \mathcal F, \phi_* \mathcal G) = D_{\mathcal C}(\mathcal F, \mathcal G)$.
\end{lemma}
\begin{proof}
The same equality is true for KL by two applications of the data processing inequality, hence it holds for $D_{\mathcal C}$.
\end{proof}

\begin{lemma}\label{lem:divergence_involution}
If there exists a measurable involution $\phi$ on $\Omega$ such that $\mathcal G = \phi_* \mathcal F$ and $\mathcal F$ is convex, then
\begin{enumerate}
  \item $D_{\mathcal C}(\mathcal F, \mathcal G) = \inf_{F \in \mathcal F}D_{\mathcal C}(F, \phi_* F)$ ; $D_{1/2}(\mathcal F, \mathcal G) = \inf_{F \in \mathcal F}D_{1/2}(F, \phi_* F)$ ;
  \item for all $F \in \mathcal F$, $D_{\mathcal C}(F, \phi_*F) = \inf_{P \in \mathcal P(\Omega), \phi_*P = P} \KL(P, F)$\\ and $D_{1/2}(F, \phi_*F) = \inf_{P \in \mathcal P(\Omega), \phi_*P = P} 2\KL(P, F)$ ;
  \item $D_{1/2}(\mathcal F, \mathcal G) = 2 D_{\mathcal C}(\mathcal F, \mathcal G)$ .
\end{enumerate}

\end{lemma}

We will use this lemma for $\Omega = \mathbb{R}$, for which $\phi$ will be the function $x \mapsto c - x$ for some $c \in \mathbb{R}$.
$\mathcal F$ and $\mathcal G$ in this case could for example be the measures with mean less than $c - y$ and larger than $c + y$ respectively.

\begin{proof}[of Lemma~\ref{lem:divergence_involution}]
1. Let $\phi$ be an  involution satisfying the hypotheses of the lemma. Let $F \in \mathcal F$ and $G \in \mathcal G$. Then by joint convexity of $D_{\mathcal C}$,
\begin{align*}
D_{\mathcal C}\left(\frac{1}{2}F + \frac{1}{2}\phi_* G, \frac{1}{2}\phi_*F + \frac{1}{2} G\right)
\le \frac{1}{2}D_{\mathcal C}(F, G) + \frac{1}{2} D_{\mathcal C}(\phi_*F, \phi_*G)
= D_{\mathcal C}(F, G)
\: .
\end{align*}
By convexity of $\mathcal F$ and $\mathcal G$, $\frac{1}{2}F + \frac{1}{2}\phi_* G \in \mathcal F$ and $\frac{1}{2}\phi_*F + \frac{1}{2} G \in \mathcal G$. Furthermore, $\phi_*\left(\frac{1}{2}F + \frac{1}{2}\phi_* G\right) = \frac{1}{2}\phi_*F + \frac{1}{2} G$. We conclude that
$D_{\mathcal C}(\mathcal F, \mathcal G) = \inf_{F \in \mathcal F} D_{\mathcal C}(F, \phi_*F)$ .
Since $D_{1/2}$ is also jointly convex, the same proof holds.

2. Let's now argue that for $F \in \mathcal F$, $D_{\mathcal C}(F, \phi_*F) = \inf_{P \in \mathcal P(\Omega), \phi_*P = P} \max\{\KL(P, F), \KL(P, \phi_*F)\} = \inf_{P \in \mathcal P(\Omega), \phi_*P = P} \KL(P, F)$. The difference between the first equality and the definition is the added condition $\phi_*P = P$. The second equality is a consequence of Lemma~\ref{lem:DC_measurable_equiv}.
\begin{align*}
\KL\left(\frac{1}{2}P + \frac{1}{2}\phi_*P, F\right)
\le \frac{1}{2}\KL(P, F) + \frac{1}{2}\KL(\phi_*P, F)
&= \frac{1}{2}\KL(P, F) + \frac{1}{2}\KL(P, \phi_*F)
\\
&\le \max\{\KL(P,F), \KL(P, \phi_*F)\}
\: , \\
\KL\left(\frac{1}{2}P + \frac{1}{2}\phi_*P, \phi_*F\right)
\le \frac{1}{2}\KL(P, \phi_*F) + \frac{1}{2}\KL(\phi_*P, \phi_*F)
&= \frac{1}{2}\KL(P, F) + \frac{1}{2}\KL(P, \phi_*F)
\\
&\le \max\{\KL(P,F), \KL(P, \phi_*F)\}
\: .
\end{align*}
We obtained that $\max\{\KL\left(\frac{1}{2}P + \frac{1}{2}\phi_*P, F\right), \KL(\frac{1}{2}P + \frac{1}{2}\phi_*P, \phi_*F)\} \le \max\{\KL(P,F), \KL(P, \phi_*F)\}$, which proves that we can indeed restrict the infimum to distributions $P$ with $\phi_*P = P$.

For $D_{1/2}$, we have $D_{1/2}(F,\phi_*F) = \inf_{P \in \mathcal P(\Omega)}(\KL(P,F) + \KL(P, \phi_*F))$. We similarly argue that we can take $P$ such that $\phi_*P=P$ : 
we also obtained above that $\KL(\frac{1}{2}P + \frac{1}{2}\phi_*P, F) + \KL(\frac{1}{2}P + \frac{1}{2}\phi_*P, \phi_*F) \le \KL(P, F) + \KL(P, \phi_*F)$.
We can thus restrict the infimum and for $P$ with $\phi_*P=P$, we have $\KL(P, F) + \KL(P, \phi_*F) = 2 \KL(P,F)$.

3. Immediate consequence of 1 and 2.
\end{proof}

\begin{theorem}\label{thm:asymptotic_testing}
Let $\mathcal F$ and $\mathcal G$ be two sets of probability distributions on $\Omega$. For $n \in \mathbb{N}$, let $E_n$ be an event on $\Omega^n$. Then
\begin{align*}
\limsup_{n \to +\infty} \frac{1}{n}\log\frac{1}{2\max\{\sup_{F \in \mathcal F} F^{\otimes n}(E_n), \sup_{G \in \mathcal G} G^{\otimes n}(E_n^c)\}}
\le D_{\mathcal C}(\mathcal F, \mathcal G) \: .
\end{align*}
\end{theorem}
This theorem is informative when the left-hand side does not go to zero. That is, it quantifies how exponentially small $\max\{\sup_{F \in \mathcal F} F^{\otimes n}(E_n), \sup_{G \in \mathcal G} G^{\otimes n}(E_n^c)\}$ can get when $n$ grows.
In order to get estimation lower bounds, we will on the other hand be interested in increasing sets of distributions $\mathcal F_n$ and $\mathcal G_n$ such that $D_{\mathcal C}(\mathcal F_n, \mathcal G_n) \to_{n \to \infty} 0$, and we will have to use other tools to get asymptotic bounds without the factor $2$ of Theorem~\ref{thm:non_asymptotic_testing}. 

\begin{proof}[of Theorem~\ref{thm:asymptotic_testing}]
Let $n \in \mathbb{N}$, $F \in \mathcal F$ and $G \in \mathcal G$. Let any $x > 0$ and $P \in \mathcal P(\Omega)$ with $P \ll F$ and $P \ll G$ such that $\log\frac{dP}{dF}$ and $\log\frac{dP}{dG}$ are $P$-integrable. If these conditions on $P$ are not satisfied, either $\KL(P,F) = +\infty$ or $\KL(P,G)=+\infty$ and then $P$ is irrelevant for the computation of $D_{\mathcal C}(F, G)$.

Let $p_{F,n}(x) = P^{\otimes n}\left( \frac{1}{n}\sum_{i=1}^n \log \frac{dP}{dF}(X_i) - \KL(P, F) > x \right)$ and define $p_{G,n}(x)$ similarly.
By Lemma~\ref{lem:change_of_measure}, Equation~\eqref{eq:kl_change_measure_3_points},
\begin{align*}
2 \max\{F^{\otimes n}(E_n), G^{\otimes n}(E_n^c)\} e^{n \max\{\KL(P, F), \KL(P, G)\} + nx}
&\ge 1 - p_{F,n}(x) - p_{G,n}(x)
\: .
\end{align*}
By the law of large numbers, for any $x > 0$, $p_{F,n}(x) \to_{n \to \infty} 0$ and $p_{G,n}(x) \to_{n \to \infty} 0$.
For $n$ large enough we can take logarithms on both sides of the inequality above and write
\begin{align*}
\log\left(2 \max\{F^{\otimes n}(E_n), G^{\otimes n}(E_n^c)\}\right) + n \max\{\KL(P, F), \KL(P, G)\} + n x
&\ge \log\left(1 - p_{F,n}(x) - p_{G,n}(x)\right)
\: .
\end{align*}
Reordering,
\begin{align*}
\frac{1}{n}\log\frac{1}{2 \max\{F^{\otimes n}(E_n), G^{\otimes n}(E_n^c)\}} + \log\left(1 - p_{F,n}(x) - p_{G,n}(x)\right) - x
&\le \max\{\KL(P, F), \KL(P, G)\}
\: .
\end{align*}
We take a limit as $n \to +\infty$
\begin{align*}
\limsup_{n \to +\infty}\frac{1}{n}\log\frac{1}{2 \max\{F^{\otimes n}(E_n), G^{\otimes n}(E_n^c)\}} - x
&\le \max\{\KL(P, F), \KL(P, G)\}
\: .
\end{align*}
Since $x > 0$ is arbitrary we get the same inequality without that $x$. The left-hand side does not depend on $P$, hence we can take an infimum on the right-hand side over $P$ and get
\begin{align*}
\limsup_{n \to +\infty}\frac{1}{n}\log\frac{1}{2 \max\{F^{\otimes n}(E_n), G^{\otimes n}(E_n^c)\}}
&\le D_{\mathcal C}(F, G)
\: .
\end{align*}

We can upper bound $F^{\otimes n}(E_n)$ by $\sup_{F \in \mathcal F}F^{\otimes n}(E_n)$ and similarly for $G$. At that point, the left-hand side depends on neither $F$ nor $G$: we can take an infimum of the right-hand side over these to get $D_{\mathcal C}(\mathcal F, \mathcal G)$.
\end{proof}

\begin{proof}[of Theorem~\ref{thm:generic_delta_bound}]
This follows from Theorem~\ref{thm:non_asymptotic_testing} applied to $\mathcal{F} = \{F^{\otimes n} \mid F \in \mathcal D,\, \mathbb{E}_F[X] \le c - x(n, \delta, G)\}$ and $\mathcal{G}=\{G^{\otimes n} \mid G \in \mathcal D,\, \mathbb{E}_G[X] \ge c + x(n, \delta, G)\}$ and $E = \{\tilde{\mu}_n \ge c\}$, using that the KL of a product of measures is the sum of the individual KLs (chain rule).
\end{proof}

\section{Proofs: lower bounds for nonparametric estimation}
\label{app:finite_variance}

\subsection{Variance}

\begin{proof}[of Lemma~\ref{lem:renyi_div_variance}]
First, the value is independent of $c$. Indeed, $\phi: x \mapsto x - c$ is an equivalence of measurable spaces between $\mathbb{R}$ and itself, and we can get that conclusion from Lemma~\ref{lem:DC_measurable_equiv}.
Let then set $c = 0$. We can apply Lemma~\ref{lem:divergence_involution} to the involution $\psi: x \mapsto -x$ to get that
\begin{align*}
&D_{1/2}\left(\left\{F \in \mathcal D_2 \mid \mathbb{E}_F[X] \le - \sqrt{2\Var_F[X] x}\right\},
  \left\{G \in \mathcal D_2 \mid \mathbb{E}_G[X] \ge \sqrt{2\Var_G[X] x}\right\}\right)
\\
&= \inf_{F \in \mathcal D_2 \mid \mathbb{E}_F[X] \le - \sqrt{2\Var_F[X] x}} D_{1/2}(F, \psi_*F) \: ,
\end{align*}
and that $D_{1/2} = 2D_{\mathcal C}$ for these sets. As a consequence of the above, we have that
\begin{align*}
&D_{1/2}\left(\left\{F \in \mathcal D_2 \mid \mathbb{E}_F[X] \le - \sqrt{2\Var_F[X] x}\right\},
  \left\{G \in \mathcal D_2 \mid \mathbb{E}_G[X] \ge \sqrt{2\Var_G[X] x}\right\}\right)
\\
&= \inf_{m > 0, \sigma > 0, m/\sigma \ge \sqrt{2 x}} D_{1/2}\left(\left\{F \mid \mathbb{E}_F[X] = -m, \Var_F[X] = \sigma^2\right\},
  \left\{G \mid \mathbb{E}_G[X] = m, \Var_G[X] = \sigma^2\right\}\right)
\: .
\end{align*}
For any $m$ and $\sigma^2$, the divergence in that last line has been computed in \citep{nishiyama2020tight}. To be exact, that paper computes the minimal Hellinger distance between the two sets, but $D_{1/2}(F, G) = -2\log(1 - \H^2(F, G))$ for any $F, G$. They also provide explicit minimizers, which we will use in a later proof. Their result is 
\begin{align*}
&\H^2\left(\left\{F \mid \mathbb{E}_F[X] = -m, \Var_F[X] = \sigma^2\right\},
  \left\{G \mid \mathbb{E}_G[X] = m, \Var_G[X] = \sigma^2\right\}\right)
\\
&= 1 - \sqrt{1 - \frac{m^2}{m^2 + \sigma^2}} \: .
\end{align*}

We get
\begin{align*}
&D_{1/2}\left(\left\{F \in \mathcal D_2 \mid \mathbb{E}_F[X] \le - \sqrt{2\Var_F[X] x}\right\},
  \left\{G \in \mathcal D_2 \mid \mathbb{E}_G[X] \ge \sqrt{2\Var_G[X] x}\right\}\right)
\\
&= \inf_{m > 0, \sigma > 0, m/\sigma \ge \sqrt{2 x}}\log \left(1 + \frac{m^2}{\sigma^2}\right)
\\
&= \log(1 + 2x) \: .
\end{align*}
Note that this also proves that 
\begin{align*}
&D_{1/2}\left(\left\{F \in \mathcal D_2 \mid \mathbb{E}_F[X] \le - \sqrt{2\Var_F[X] x}\right\},
  \left\{G \in \mathcal D_2 \mid \mathbb{E}_G[X] \ge \sqrt{2\Var_G[X] x}\right\}\right)
\\
&= D_{1/2}\left(\left\{F \mid \mathbb{E}_F[X] = -\sqrt{2 x}, \Var_F[X] = 1\right\},
  \left\{G \mid \mathbb{E}_G[X] = \sqrt{2 x}, \Var_G[X] = 1\right\}\right)
\: .
\end{align*}

\end{proof}

\begin{proof}[of Theorem~\ref{thm:lower_bound_finite_var}]
The first inequality is an immediate consequence of Theorem~\ref{thm:generic_delta_bound} and Lemma~\ref{lem:renyi_div_variance}.

Let's prove the second inequality. As in the proof of the previous lemma, we first remark that
\begin{align*}
&D_{1/2}\left(\left\{F \in \mathcal D_2 \mid \mathbb{E}_F[X] \le - \sqrt{2\Var_F[X] y(n, \delta)}\right\},
  \left\{G \in \mathcal D_2 \mid \mathbb{E}_G[X] \ge \sqrt{2\Var_G[X] y(n, \delta)}\right\}\right)
\\
&= D_{1/2}\left(\left\{F \mid \mathbb{E}_F[X] = -\sqrt{2 y(n, \delta)}, \Var_F[X] = 1\right\},
  \left\{G \mid \mathbb{E}_G[X] = \sqrt{2 y(n, \delta)}, \Var_G[X] = 1\right\}\right)
\: .
\end{align*}
The explicit minimizers for the Rényi divergence between these two sets have been computed in \citep{nishiyama2020tight}. We use them and thus define $u_n = \sqrt{1 + 2 y(n, \delta)}$ and three distributions $P_n, F_n$ and $G_n$ supported on $u_n$ and $-u_n$, with $P_n(\{u_n\}) = 1/2$ , $F_n(\{u_n\}) = \frac{1}{2} - \frac{1}{2}\sqrt{\frac{2y(n, \delta)}{1 + 2y(n, \delta)}}$ and $G_n(\{u_n\}) = 1 - F_n(u_n)$.

$F_n$ and $G_n$ have mean $-\sqrt{2 y(n, \delta)}$ and $\sqrt{2 y(n, \delta)}$ respectively, and variance 1. Hence, $F_n \in \left\{F \in \mathcal D_2 \mid \mathbb{E}_F[X] \le -\sqrt{2\Var_F[X] y(n, \delta)}\right\}$ and $G_n \in \left\{G \in \mathcal D_2 \mid \mathbb{E}_G[X] \ge \sqrt{2\Var_G[X] y(n, \delta)}\right\}$~.
The log-likelihood ratio $\log\frac{dP_n}{dF_n}$ is supported on $\log (\sqrt{1 + 2y(n, \delta)}) \pm \log(\sqrt{1 + 2y(n, \delta)} + \sqrt{2 y(n, \delta)})$, with probability $1/2$ for each point.
$\log\frac{dP_n}{dG_n}$ has the same distribution. $\KL(P_n, F_n) = \log (\sqrt{1 + 2y(n, \delta)})$ and the variance of $\log\frac{dP_n}{dF_n}$ is $V_n = \log^2(\sqrt{1 + 2y(n, \delta)} + \sqrt{2 y(n, \delta)})$. We denote $\KL(P_n, F_n)$ (equal to $\KL(P_n, G_n)$) by $\KL_n$.

By Lemma~\ref{lem:change_of_measure}, Equation~\eqref{eq:kl_change_measure_3_points}, for any $x_n > 0$,
\begin{align*}
2 \delta e^{n \max\{\KL(P_n, F_n), \KL(P_n, G_n)\} + n x_n}
&\ge 1 - P_n\left( \frac{1}{n}\sum_{i=1}^n \log \frac{dP_n}{dF_n}(X_i) - \KL(P_n, F_n) > x_n \right)
  \\&\quad - P_n\left( \frac{1}{n}\sum_{i=1}^n \log \frac{dP_n}{dG_n}(X_i) - \KL(P_n, G_n) > x_n \right)
\: .
\end{align*}
Using Chebyshev's inequality, we get $2 \delta e^{n \KL_n + n x_n} \ge 1 - \frac{2V_n}{n x_n^2}$ . We then choose $x_n = \sqrt{\frac{4 V_n}{n}}$, reorder the equation and obtain $\frac{1}{n}\log \frac{1}{4 \delta} \le \KL_n + 2\sqrt{\frac{V_n}{n}}$ , which is the second bound of the theorem.
\end{proof}

\paragraph{Estimator}

\begin{lemma}\label{lem:E1}\cite[Lemma E.1]{agrawal2021optimal}
Let $\Lambda \subseteq \mathbb{R}^d$ be a compact and convex set and let $q$ be the uniform distribution on $\Lambda$. Let $(g_n)_{n \in \mathbb{N}}$ be exp-concave functions $\Lambda \to \mathbb{R}$. Then
\begin{align*}
\max_{\lambda \in \Lambda}\sum_{n=1}^N g_n(\lambda) \le \log \mathbb{E}_{\lambda \sim q}\left[ e^{\sum_{n=1}^N g_n(\lambda)} \right] + d \log (N + 1) + 1 \: .
\end{align*}
\end{lemma}

\begin{proof}[of Lemma~\ref{lem:Kinf_concentration}]
Since we can shift both distributions in the KL we can w.l.o.g. consider the case where $P$ has mean 0 and $0 \le m$.
\begin{align*}
\inf_{F, \mathbb{E}_F[X] \le m, \Var_F[X] \le \sigma^2}\KL(\hat{P}_n, F)
\le \inf_{F, \mathbb{E}_F[X] \le 0, \Var_F[X] \le \sigma^2}\KL(\hat{P}_n, F) \: .
\end{align*}
Then we get rid of the dependence of the variance constraint on the mean of the distribution to obtain two linear constraints on $F$,
\begin{align*}
\inf_{F, \mathbb{E}_F[X] \le 0, \Var_F[X] \le \sigma^2}\KL(\hat{P}_n, F)
&= \inf_{\mu \le 0}\inf_{F, \mathbb{E}_F[X] = \mu, \mathbb{E}_F[X^2] \le \sigma^2 + \mu^2}\KL(\hat{P}_n, F)
\\
&\le \inf_{\mu \le 0}\inf_{F, \mathbb{E}_F[X] = \mu, \mathbb{E}_F[X^2] \le \sigma^2}\KL(\hat{P}_n, F)
\\
&= \inf_{F, \mathbb{E}_F[X] \le 0, \mathbb{E}_F[X^2] \le \sigma^2}\KL(\hat{P}_n, F)
\: .
\end{align*}
This infimum of a KL has a dual formulation as a supremum \citep{agrawal2021optimal,honda2010asymptotically,honda2011asymptotically,garivier2018kl}.
\begin{align*}
\inf_{F, \mathbb{E}_F[X] \le 0, \mathbb{E}_F[X^2] \le \sigma^2}\KL(\hat{P}_n, F)
&= \frac{1}{n} \sup_{\lambda \in \Lambda} \sum_{i=1}^n \log \left( 1 + \lambda_1 X_i + \lambda_2(X_i^2 - \sigma^2) \right)
\end{align*}
where $\Lambda = \{\lambda \in \mathbb{R}_+^2 \mid \forall x \in \mathbb{R},\ 1 + \lambda_1 x + \lambda_2(x^2 - \sigma^2) \ge 0 \}$ is a compact set.
We can then apply Lemma~\ref{lem:E1} with $d = 2$, that set $\Lambda$ and $g_n(\lambda) = \log \left( 1 + \lambda_1 X_n + \lambda_2(X_n^2 - \sigma^2) \right)$.
\begin{align*}
&\max_{\lambda \in \Lambda} \sum_{i=1}^n \log(1 + \lambda_1 X_i + \lambda_2(X_i^2 - \sigma^2))
\\
&\le \log \mathbb{E}_{\lambda \sim q}\left[ \prod_{i=1}^n (1 + \lambda_1 X_i + \lambda_2(X_i^2 - \sigma^2)) \right] + 2 \log (n + 1) + 1 \: .
\end{align*}
Finally, under $P$, $\mathbb{E}_{\lambda \sim q}\left[ \prod_{i=1}^n (1 + \lambda_1 X_i + \lambda_2(X_i^2 - \sigma^2)) \right]$ is a non-negative supermartingale with expectation at most 1, which means that with probability $1 - \delta$ it is less than $1/\delta$.

\end{proof}
\begin{proof}[of Theorem~\ref{thm:estimator_upper_bound}]
We first remark that for all $c \in \mathbb{R}$,
\begin{align*}
\max\{\hat{d}_{L,n}(c), \hat{d}_{R,n}(c)\}
&= \max\{\inf_{F \in \mathcal F_c}\KL(\hat{P}_n, F), \inf_{G \in \mathcal G_c}\KL(\hat{P}_n, G)\}
\\
&\ge \min_{c \in \mathbb{R}} \inf_{P \in \mathcal{P}(\mathbb{R})} \max\{\inf_{F \in \mathcal F_c}\KL(P, F), \inf_{G \in \mathcal G_c}\KL(P, G)\}
\\
&= \frac{1}{2}\log(1 + 2 y(n, \delta))
\: .
\end{align*}

Suppose that on distribution $F$ with mean $m_F$ and variance $\sigma_F^2$, the estimator $\tilde{\mu}_n$ makes an error larger than $\sqrt{2\sigma_F^2 y(n, \delta)}$. Suppose for example that $\tilde{\mu}_n \ge m_F + \sqrt{2\sigma_F^2 y(n, \delta)}$.
Then by definition of $\tilde{\mu}_n$, $d_{L,n}(m_F + \sqrt{2\sigma_F^2 y(n, \delta)}) = \max\{d_{L,n}(m_F + \sqrt{2\sigma_F^2 y(n, \delta)}), d_{R,n}(m_F + \sqrt{2\sigma_F^2 y(n, \delta)})\}$.
We have then
\begin{align*}
\frac{1}{2}\log(1 + 2 y(n, \delta))
&\le \max\{d_{L,n}(m_F + \sqrt{2\sigma_F^2 y(n, \delta)}), d_{R,n}(m_F + \sqrt{2\sigma_F^2 y(n, \delta)})\}
\\
&= d_{L,n}(m_F + \sqrt{2\sigma_F^2 y(n, \delta)})
\\
&= \inf_{G, \mathbb{E}_G[X] \le m_F + \sqrt{2\sigma_F^2 y(n, \delta)} - \sqrt{2\Var_G[X] y(n, \delta)}} \KL(\hat{F}_n, G)
\\
&\le \inf_{G, \Var_G[X] \le \sigma_F^2, \mathbb{E}_G[X] \le m_F + \sqrt{2\sigma_F^2 y(n, \delta)} - \sqrt{2\Var_G[X] y(n, \delta)}} \KL(\hat{F}_n, G)
\\
&\le \inf_{G, \Var_G[X] \le \sigma_F^2, \mathbb{E}_G[X] \le m_F} \KL(\hat{F}_n, G)
\: .
\end{align*}
We proceed similarly in the case $\tilde{\mu}_n \le m_F - \sqrt{2\sigma_F^2 y(n, \delta)}$, but using $\hat{d}_{R,n}$ instead of $\hat{d}_{L,n}$, to obtain $\frac{1}{2}\log(1 + 2 y(n, \delta)) \le \inf_{G, \Var_G[X] \le \sigma_F^2, \mathbb{E}_G[X] \ge m_F} \KL(\hat{F}_n, G)$.

By the concentration lemma~\ref{lem:Kinf_concentration}, with probability $1 - \delta$, if $\tilde{\mu}_n$ makes an error larger than $\sqrt{2\sigma_F^2 y(n, \delta)}$ then
\begin{align*}
\inf_{G, \Var_G[X] \le \sigma_F^2, \mathbb{E}_G[X] \le m_F} \KL(\hat{F}_n, G) &< \frac{1}{n}\log\frac{2e(n+1)^2}{\delta}
\: , \\
\inf_{G, \Var_G[X] \le \sigma_F^2, \mathbb{E}_G[X] \ge m_F} \KL(\hat{F}_n, G) &< \frac{1}{n}\log\frac{2e(n+1)^2}{\delta}
\: .
\end{align*}
Then with probability $1 - \delta$, if there is a large error then $\log(1 + 2 y(n, \delta)) < \frac{2}{n}\log\frac{2e(n+1)^2}{\delta}$. This is a contradiction with the definition of $y(n, \delta)$. We conclude that with probability $1 - \delta$, the error is smaller than $\sqrt{2\sigma_F^2 y(n, \delta)}$.
\end{proof}

\subsection{Moment $1 + \alpha$}

\begin{proof}[of Theorem~\ref{thm:lower_bound_finite_1_alpha}]
Let $p_n \in (0,1)$ and $c > 0$ and consider the distributions $P_n = \delta_0$, $G_n = (1-p_n)\delta_0 + p_n \delta_{c}$, $F_n = (1-p_n) \delta_0 + p_n \delta_{-c}$. \cite{devroye2016sub} use distributions of this form in their lower bound.

By Lemma~\ref{lem:change_of_measure}, Equation~\eqref{eq:kl_change_measure_3_points}, for any $x_n > 0$, if $F_n \in \{F \in \mathcal D_{1+\alpha} \mid \mathbb{E}_{F}[X] \le - \varepsilon_n M_{1+\alpha}(F)\}$ and $G_n \in \{F \in \mathcal D_{1+\alpha} \mid \mathbb{E}_{F}[X] \ge \varepsilon_n M_{1+\alpha}(F)\}$ then
\begin{align*}
2 \delta e^{n \max\{\KL(P_n, F_n), \KL(P_n, G_n)\} + n x_n}
&\ge 1 - P_n\left( \frac{1}{n}\sum_{i=1}^n \log \frac{dP_n}{dF_n}(X_i) - \KL(P_n, F_n) > x_n \right)
  \\&\quad - P_n\left( \frac{1}{n}\sum_{i=1}^n \log \frac{dP_n}{dG_n}(X_i) - \KL(P_n, G_n) > x_n \right)
\: .
\end{align*}
Since $\log\frac{dP_n}{dF_n}$ and $\log\frac{dP_n}{dG_n}$ are Dirac distributions at the point $\KL_n := \KL(P_n, F_n) = \KL(P_n, G_n) = - \log(1 - p_n)$, the two probabilities on the right are equal to 0 and the inequality reduces to
\begin{align*}
2 \delta e^{n \KL_n + n x_n}
&\ge 1
\: .
\end{align*}
We get $\frac{1}{n}\log\frac{1}{2 \delta} \le -\log(1 - p_n)$. We still need to ensure that $F_n \in \{F \in \mathcal D_{1+\alpha} \mid \mathbb{E}_{F}[X] \le - \varepsilon_n M_{1+\alpha}(F)\}$ and $G_n \in \{F \in \mathcal D_{1+\alpha} \mid \mathbb{E}_{F}[X] \ge \varepsilon_n M_{1+\alpha}(F)\}$.
Remark that we have  $\E_{G_n}[X] = p_n c = -\E_{F_n}[X]$ . The moment $1+\alpha$ of these distributions is
\begin{align*}
\left(M_{1+\alpha}(F_n)\right)^{1+\alpha}
= \left(M_{1+\alpha}(G_n)\right)^{1+\alpha}
&= c^{1+\alpha} p_n (1-p_n)\left(p_n^\alpha+(1-p_n)^{\alpha}\right)
\: .
\end{align*}
Let $\eta_n = (1-p_n)\left(p_n^\alpha+(1-p_n)^{\alpha}\right)$. For all $n$, $\eta_n \le 2$. Also, $\eta_n = 1 + o_{p_n \to 0}(1)$.

We thus need
\begin{align*}
p_n c \ge \varepsilon_n c p_n^{1/(1+\alpha)} \eta_n^{1/(1+\alpha)} \: .
\end{align*}
Let us choose $p_n = \varepsilon_n^{\frac{1 + \alpha}{\alpha}} 2^{1/\alpha}$, which satisfies this inequality. Using $-\log(1 - p_n) \le p_n$, we obtain the lower bound
\begin{align*}
\frac{1}{2^{\frac{1}{1 + \alpha}}}\left(\frac{1}{n}\log\frac{1}{2 \delta}\right)^{\frac{\alpha}{1 + \alpha}}
\le \varepsilon_n
\: .
\end{align*}
If $\log(1/\delta)/n \to 0$, then we can take $p_n \to 0$, such that $\eta_n = 1 + o(1)$ and the constant 2 is replaced by $1 + o(1)$.
\end{proof}

\subsection{Log-Lipschitz condition}

\begin{proof}[of Theorem~\ref{thm:non_asymptotic_lip}]
Let $h \in \R$. By inequality between KL and $\chi^2$ distance (see for instance \citep{gibbs2002choosing}), we have
\begin{equation}\label{eq:KL_chi2}
\KL(F(x), F(x+h))\le \int\frac{(f(x+h)-f(x))^2}{f(x)}\d x.
\end{equation}
Then, using that  $\log(f)$ is Lipschitz, we have:\\
$\bullet$ If $f(x+h)\ge f(x)$, then by Lipschitz property $\log(f(x+h))-\log(f(x)) \le Lh$ hence $f(x+h) \le f(x)e^{Lh}$ and 
$$\frac{f(x+h)-f(x)}{f(x)}\le \frac{e^{Lh}f(x)-f(x)}{f(x)}= e^{Lh}-1\,.$$
$\bullet$ If $f(x+h)\le  f(x)$, then by Lipschitz property $\log(f(x))-\log(f(x+h)) \le Lh$ hence $f(x) \le f(x+h)e^{Lh}$ and 
$$\frac{f(x)-f(x+h)}{f(x)}\le \frac{e^{Lh}f(x+h)-f(x+h)}{f(x)}= (e^{Lh}-1)\frac{f(x+h)}{f(x)}\le e^{Lh}-1\,.$$
In both cases, $\frac{|f(x+h)-f(x)|}{f(x)}\le e^{Lh}-1$. Injecting this in Equation~\eqref{eq:KL_chi2} and using the formula for the total variation distance as the $L^1$ distance between the densities, we get
\begin{align*}
\KL(F(x), F(x+h))\le  (e^{Lh}-1)\int|f(x+h)-f(x)|\d x  =(e^{Lh}-1)\mathrm{TV}(F(x), F(x+h))\,.
\end{align*}
Then, by Pinsker's inequality
\begin{align*}
\KL(F(x), F(x+h))\le (e^{Lh}-1)\sqrt{\frac{1}{2}\KL(F(x), F(x+h))}\,,
\end{align*}
which implies
$$ \KL(F(x), F(x+h))\le \frac{1}{2}(e^{Lh}-1)^2.$$
Now, it remains to apply Theorem~\ref{thm:generic_delta_bound} (with the Chernoff divergence) with $G_+ = F(x+\varepsilon_n)$ and $G_- = F(x-\varepsilon_n)$, we have
\begin{align*}
\frac{1}{n}\log \frac{1}{4\delta} \le (e^{L\varepsilon_n}-1)^2
\: .
\end{align*}
Isolate $\varepsilon_n$ to conclude.
\end{proof}

\section{Proof of Lemma~\ref{lem:examples}}\label{sec:proof_lem_example}
\textbf{Gaussian distributions}
Let $\mathcal{D} = \{\mathcal{N}(\mu,\sigma^2), \mu\in \R, \sigma>0 \}$. Using the formula for Hellinger distance between Gaussian distributions
\begin{align*}
D_{1/2}&\left(\left\{F \in \mathcal D \mid \mathbb{E}_F[X] \le c - \sqrt{2\sigma_F^2 y(n, \delta)}\right\}, \left\{G \in \mathcal D \mid \mathbb{E}_G[X] \ge c + \sqrt{2\sigma_G^2 y(n, \delta)}\right\}\right)\\
&= \inf_{\mu_1 \le c - \sqrt{2\sigma_1^2 y(n, \delta)}, \sigma_1>0}\inf_{\mu_2\ge c + \sqrt{2\sigma_2^2 y(n, \delta)}, \sigma_2>0} D_{1/2}\left(\mathcal{N}(\mu_1,\sigma_1^2), \mathcal{N}(\mu_2,\sigma_2^2)\right)\\
&= \inf_{\mu_1 \le c - \sqrt{2\sigma_1^2 y(n, \delta)}, \sigma_1>0}\inf_{\mu_2\ge c + \sqrt{2\sigma_2^2 y(n, \delta)}, \sigma_2>0} - 2 \log\left( \sqrt{\frac{2\sigma_1\sigma_2}{\sigma_1^2+\sigma_2^2}}e^{-\frac{1}{4}\frac{(\mu_1-\mu_2)^2}{\sigma_1^2+\sigma_2^2}}\right).
\end{align*}
This last term depends on $\mu_1$ and $\mu_2$ only through their distance $|\mu_1-\mu_2|$ hence there is in fact no dependency in $c$ and the infimum over $\mu_1,\mu_2$ is obtained for $|\mu_1-\mu_2|= \sqrt{2\sigma_1^2 y(n, \delta)} +\sqrt{2\sigma_2^2 y(n, \delta)}  $. Then, 
\begin{align*}
D_{1/2}&\left(\{G \in \mathcal D \mid \mathbb{E}_G[X] \le c - \sqrt{2\sigma_G^2 y(n, \delta)}\}, \{G \in \mathcal D \mid \mathbb{E}_G[X] \ge c + \sqrt{2\sigma_G^2 y(n, \delta)}\}\right)\\
&= \inf_{\sigma_1,\sigma_2>0}- 2 \log\left(\sqrt{\frac{2\sigma_1\sigma_2}{\sigma_1^2+\sigma_2^2}}e^{-\frac{(\sigma_1+\sigma_2)^2}{2(\sigma_1^2+\sigma_2^2)} y(n, \delta)}\right)\\
\end{align*}
Then, having $2\sigma_1\sigma_2 \le \sigma_1^2+\sigma_2^2$ with equality if and only if $\sigma_1=\sigma_2$, we get
\begin{align*}
D_{1/2}&\left(\{G \in \mathcal D \mid \mathbb{E}_G[X] \le c - \sqrt{2\sigma_G^2 y(n, \delta)}\}, \{G \in \mathcal D \mid \mathbb{E}_G[X] \ge c + \sqrt{2\sigma_G^2 y(n, \delta)}\}\right)\\
&= - 2 \log\left(e^{- y(n, \delta)}\right)= 2y(n,\delta).
\end{align*}

From Theorem~\ref{thm:generic_delta_bound}, we obtain that $\frac{\log(1/4\delta)}{n}\le 2y(n,\delta)$ which concludes this part of the proof.

\textbf{Laplace distributions}
Let $\mathcal{D} = \{\mathrm{Laplace}(\mu,b), \mu\in \R, b>0 \}$ where $F \sim \mathrm{Laplace}(\mu,b)$ has density $f(x)= \frac{1}{2b}\exp(-|x-\mu|/b)$ with variance $2b^2$. We have
\begin{align*}
D_{\mathcal C}&\left(\left\{F \in \mathcal D \mid \mathbb{E}_F[X] \le c - \sqrt{2\sigma_F^2 y(n, \delta)}\right\}, \left\{G \in \mathcal D \mid \mathbb{E}_G[X] \ge c + \sqrt{2\sigma_G^2 y(n, \delta)}\right\}\right)\\
&= \inf_{P \in \mathcal{P}} \max\left\{\inf_{\mu_1 \le c - 2\sqrt{b_1^2 y(n, \delta)}, b_1>0 } \KL(P,\mathrm{Laplace}(\mu_1,b_1)),\inf_{\mu_2 \ge c - 2\sqrt{b_2^2 y(n, \delta)},b_2>0 } \KL(P,\mathrm{Laplace}(\mu_2,b_2)) \right\} 
\end{align*}
Now, choose $P = \mathrm{Laplace}(c,b)$ for some $b>0$, having that the KL between Laplace distribution is invariant by shift of the two distributions, using the formula for KL between two Laplace distributions from~\cite[Appendix A.]{Meyer_2021_CVPR}, we have that for this choice of $P$, the two infimum of KL are equal and then
\begin{align*}
D_{\mathcal C}&\left(\left\{F \in \mathcal D \mid \mathbb{E}_F[X] \le c - \sqrt{2\sigma_F^2 y(n, \delta)}\right\}, \left\{G \in \mathcal D \mid \mathbb{E}_G[X] \ge c + \sqrt{2\sigma_G^2 y(n, \delta)}\right\}\right)\\
&\le   \inf_{ b_1>0 } \KL\left(\mathrm{Laplace}(c,b_1),\mathrm{Laplace}\left(c+2\sqrt{b_1^2 y(n, \delta)},b_1\right)\right)\\
&=  \inf_{ b_1>0 } \left(\frac{b}{b_1}e^{-2\sqrt{b_1^2 y(n, \delta)}/\sigma}+\frac{2\sqrt{b_1^2 y(n, \delta)}}{\sigma}+\log\left(\frac{b_1}{b}\right)-1\right).
\end{align*}
In particular, choosing $\sigma_1=\sigma$, we get
\begin{align*}
D_{\mathcal C}&\left(\left\{F \in \mathcal D \mid \mathbb{E}_F[X] \le c - \sqrt{2\sigma_F^2 y(n, \delta)}\right\}, \left\{G \in \mathcal D \mid \mathbb{E}_G[X] \ge c + \sqrt{2\sigma_G^2 y(n, \delta)}\right\}\right)\\
&\le    e^{-2\sqrt{y(n, \delta)}}+2\sqrt{ y(n, \delta)}-1.
\end{align*}
Choosing $y(n,\delta)= \frac{2}{n}\log\left(\frac{1}{4\delta}\right)$, from Theorem~\ref{thm:generic_delta_bound} we conclude that $y(n,\delta)$ satisfies the following equation:
$$\frac{1}{4}y(n,\delta) \le e^{-2\sqrt{y(n, \delta)}}+2\sqrt{ y(n, \delta)}-1 $$
The function $x \mapsto x^2/4 - e^{-2x}- 2x + 1$ is non-increasing over $[0,\log(8)/2]$ and convex for $x\ge \log(8)/2$, it is equal to $0$ at $0$ and positive for $x=8$, hence necessarily $y(n,\delta)= \frac{2}{n}\log\left(\frac{1}{4\delta}\right)\le 8$. Solve this inequality in $\delta$ to conclude.

\section{Proofs: lower bounds in semi-parametric families}
\label{app:parametric}

\begin{proof}[of Theorem~\ref{thm:without_2}]
Let $\theta \in \Theta$ be such that $0<\I(P_\theta)<\infty$.
We first reduce to testing whether the parameter is higher of lower than $\theta$ and consider two distributions $P_{\theta-h/\sqrt{n}}$ and $P_{\theta+h/\sqrt{n}}$.
Then $\delta \ge \max\{P_{\theta-h/\sqrt{n}}(\hat{\theta}_n \ge c), P_{\theta+h/\sqrt{n}}(\hat{\theta}_n < c)\}$, where $\hat{\theta}_n$ is the estimator.

By Lemma~\ref{lem:change_of_measure}, equation~\eqref{eq:kl_change_measure_3_points}, for any $x>0$,
\begin{align}
&2 \delta e^{n \max\{\KL(P_\theta, P_{\theta - \eta/\sqrt{n}}), \KL(P_\theta, P_{\theta + \eta/\sqrt{n}})\} + \eta\sqrt{\I(P_\theta)} x}\nonumber
\\
&\ge 1 - P_\theta\left( \frac{1}{\eta\sqrt{\I(P_\theta)}}\sum_{i=1}^n \left(\log \frac{p_\theta}{p_{\theta - \eta/\sqrt{n}}}(X_i) - \KL(P_\theta, P_{\theta - \eta/\sqrt{n}})\right) > x \right)\nonumber
  \\&\qquad - P_\theta\left( \frac{1}{\eta\sqrt{\I(P_\theta)}}\sum_{i=1}^n \left(\log \frac{p_\theta}{p_{\theta +\eta/\sqrt{n}}}(X_i) - \KL(P_\theta, P_{\theta + \eta/\sqrt{n}})\right) > x \right)
\: . \label{eq:without_2_ineq}
\end{align}
From Theorem~\ref{thm:clt_qmd}, $\frac{1}{\eta\sqrt{\I(P_\theta)}}\sum_{i=1}^n \left(\log \frac{p_\theta}{p_{\theta - \eta/\sqrt{n}}}(X_i) - \KL(P_\theta, P_{\theta - \eta/\sqrt{n}})\right) \xrightarrow[n \to \infty]{d} \mathcal{N}(0, 1)$.

The right-hand side of \eqref{eq:without_2_ineq} tends to $1 - 2\Phi(-x)$ . The limit of the left-hand side is $2 \delta e^{\frac{1}{2}\eta^2 \I(P_\theta) + \eta\sqrt{\I(P_\theta)} x}$. We obtained
\begin{align*}
\log \frac{1 - 2 \Phi(-x)}{2\delta} \le \frac{1}{2}\eta^2 \I(P_\theta) + \eta \sqrt{\I(P_\theta)} x
\end{align*}
We then minimize over $\theta$ and set $x = 1$, for which $1 - 2 \Phi(-x) \ge 1/2$. We get $\log \frac{1}{4 \delta} \le \frac{1}{2}\eta^2 \uI + \eta \sqrt{\uI}$.
The first statement of the theorem is obtained by solving that quadratic equation in $\eta$.

Let's prove the second result: by the first result, for all $\delta$ we have $\eta_\delta \ge \frac{1}{\sqrt{\uI}}\left(\sqrt{1 + 2 \log\frac{1}{4 \delta}} - 1\right)$.
As $\delta \to 0$, the right-hand side is equivalent to $\sqrt{\frac{2\log(1/\delta)}{\uI}}$.
\end{proof}

\begin{proof}[of Theorem~\ref{thm:local_lowerbound_fisher}]
Let us prove that $\{F_{Y_h}, h \neq 0\}$ is q.m.d. at any $h \neq 0$. By direct differentiation, we have that $h \mapsto f_{Y_h}(x)$ is differentiable for almost all $h$ and with derivative
$$\frac{\partial f_{Y_h}(x)}{\partial h} =  \frac{-1}{h^2}\int_{0}^hf(x-2u)\d u + \frac{1}{h}f(x-2h).$$
this derivative is continuous in any $h \notin (0,a/2,b/2)$. Fisher's information is then 
$$\I(F_{Y_h}) = 4\int \left(\frac{\frac{-1}{h^2}\int_{0}^hf(x-2u)\d u + \frac{1}{h}f(x-2h)}{f_{Y_h}(x)}\right)^2f_{Y_h}(x)\d x.$$
and it is continuous in all $h \notin (0,a/2,b/2)$ as an integral of continuous functions.
Then, from Theorem~\ref{thm:condition q.m.d}, this implies that the family $\{F_{Y_h}, h \notin (0,a/2,b/2)\}$ is q.m.d at any $h \notin (0,a/2,b/2)$.

Then, from Theorem~\ref{thm:without_2}, we have  
\begin{align*}
h \ge \frac{1}{\inf_{a>0}\sqrt{\I(F_{Y_a})}}\left(\sqrt{1+2\log\frac{1}{4\delta}}-1\right)
\end{align*}
In particular, we have that 
\begin{align*}
\inf_{a>0}\I(F_{Y_a}) &\le \lim_{a \to 0}\I(F_{Y_a}) \\
&= \lim_{a \to 0, a>0}\int \left(\frac{f(x-2a)-\frac{1}{a}\int_{0}^af(x-2u)\d u}{af_{Y_a}(x)}\right)^2f_{Y_a}(x)\d x\\
&= \lim_{a \to 0, a>0}\int \left(\frac{f(x-2a)-\int_{0}^1f(x-2au)\d u}{af_{Y_a}(x)}\right)^2f_{Y_a}(x)\d x\\
&= \lim_{a \to 0, a>0}\int \left(\frac{\int_{0}^1\frac{f(x-2a)-f(x-2au)}{a}\d u}{f_{Y_a}(x)}\right)^2f_{Y_a}(x)\d x\\
&= \int_a^b \left(\frac{f'(x)}{f(x)}\right)^2f(x)\d x = \I(F).
\end{align*}
Which concludes the proof.
\end{proof}

\begin{proof}[of Corollary~\ref{cor:local_lowerbound_fisher_location}]
This follows from Theorem~\ref{thm:local_lowerbound_fisher} and the fact that location models $\{F(\cdot +h)\mid h \in \R\}$ are q.m.d. Indeed, because $f$ is absolutely continuous and $\I(F)<\infty$, by \citep[Lemma A.1]{hajek1972local} we have that $\sqrt{f}$ is also absolutely continuous, and by Theorem~\ref{thm:condition q.m.d}, $\{F(x-h) |h \in \R\}$ is q.m.d. at any $h \in \R$. We conclude using that $\uI$ is equal to $\I(F)$ because Fisher's information is invariant by shift of the distribution.
\end{proof}

\begin{proof}[of Corollary~\ref{cor:lower_bound_semi_bounded}]
Because $\widehat{\mu}_n$ is a $\left(\eta/\sqrt{n}, \delta\right)$-estimator on the set of distribution supported on $[0,\infty)$, it is in particular a $\left(\eta/\sqrt{n}, \delta\right)$-estimator for $F$ that minimizes the Fisher information among all distributions with support on $[0,\infty)$ and variance $\sigma_F^2$. By \cite[Proposition 3]{bercher2009minimum}, we have that the distribution that minimizes the Fisher information is a $\chi^2$ distribution (which is absolutely continuous on $\R_+$) and the resulting fisher information is $\I(F)=(9-24/\pi)/\sigma_F^2$. The result follows by injecting this value in Theorem~\ref{thm:local_lowerbound_fisher}.
\end{proof}

\begin{proof}[of Corollary~\ref{cor:lower_bound_bounded}]
Because $\widehat{\mu}_n$ is a $\left(\eta/\sqrt{n}, \delta\right)$-estimator on the set of distribution supported on $[a,b]$, it is in particular a $\left(\eta/\sqrt{n}, \delta\right)$-estimator for $F$ that minimizes the Fisher information among all distributions with support on $[a,b]$. By invariance by translation and transformation by scaling $\I(F)= \I(F_s)$ where $F_s$ is a distribution on $[0,1]$. By \cite[Proposition 7]{bercher2009minimum}, we have that the distribution that minimizes the Fisher information over distributions on $[0,1]$ is $f(x)=\cos^2(\pi x/2)$ (which is absolutely continuous on $[0,1]$) with variance $\sigma^2 = \frac{1}{3}- \frac{2}{\pi}$ and the resulting fisher information is $\I(F)=\pi^2$. The result follows by injecting this value in Theorem~\ref{thm:local_lowerbound_fisher}.
\end{proof}
\section{Background on q.m.d. families of distributions}\label{sec:app_qmd}

\begin{theorem}[Sufficient condition for q.m.d.]\label{thm:condition q.m.d}
Let $\theta_0 \in \Theta\subset\R$. Assume the following conditions:
\begin{itemize}
\item $\sqrt {p_\theta(x)}$ is an absolutely continuous function of $\theta$ in some neighborhood of $\theta_0$ for $\mu$-almost all $x$. 
\item For $\mu$-almost all $x$, the derivative $p_\theta'(x)= \frac{\partial}{\partial \theta} p_\theta(x)$ exists at $\theta = \theta_0$.
\item The Fisher information $\I(F_\theta)$ is finite and continuous at $\theta=\theta_0$.
\end{itemize}
Then $\{P_\theta, \quad \theta \in \Theta\}$ is q.m.d. at $\theta_0$ with quadratic mean derivative $\eta(x,\theta)$ given by 
$$\eta(x,\theta):=
\begin{cases}p'_\theta(x)/(2\sqrt{p_\theta(x)}) & \text{ if }p_\theta(x)>0 \text{ and } p_\theta'(x) \text{ exists}\\
0 & \text{otherwise}.
\end{cases}
$$
\end{theorem}
For completeness, we remind that $g(\theta)$ is an absolutely continuous function on an interval $[a,b]$, for $a,b\in \R$ if for any $\theta \in [a,b]$, $g(\theta)=g(a)+\int_a^\theta h(x)\d x$ for some integrable function $h$. In particular, if $g$ is pointwise continuously derivable the $g$ is absolutely continuous.
\end{document}